\documentclass[pdflatex,sn-mathphys-num]{sn-jnl}

\usepackage[utf8]{inputenc}
\usepackage[T1]{fontenc}
\usepackage{amsmath,amssymb,amsthm,mathtools}
\usepackage{bm}
\usepackage{enumitem}
\usepackage{graphicx}
\usepackage{booktabs}
\usepackage{caption}

\usepackage{xcolor} 

\hypersetup{hidelinks} 

\newtheorem{theorem}{Theorem}[section]
\newtheorem{proposition}[theorem]{Proposition}
\newtheorem{lemma}[theorem]{Lemma}

\theoremstyle{definition}
\newtheorem{problem}[theorem]{Problem}

\theoremstyle{remark}
\newtheorem{remark}[theorem]{Remark}

\newcommand{\R}{\mathbb{R}}
\newcommand{\dd}{\,\mathrm{d}}

\begin{document}

\title{\bfseries Age-Structured Harvesting Models: A Structural Comparison of Rate-Control and Effort-Control Optimality Systems}

\title{\bfseries Age-Structured Harvesting Models: A Structural Comparison of Rate-Control and Effort-Control Optimality Systems}

\author[1,2]{\fnm{Jiguang} \sur{Yu}}
\equalcont{All authors contributed equally as co-first authors.}

\author*[3]{\fnm{Louis Shuo} \sur{Wang}}\email{wang.s41@northeastern.edu; ORCID: 0009-0007-0977-7894}
\equalcont{All authors contributed equally as co-first authors.}

\author*[4]{\fnm{Ye} \sur{Liang}}\email{ye-liang@uiowa.edu; ORCID: 0009-0000-2220-319X}
\equalcont{These authors contributed equally to this work as co-first authors.}




\affil[1]{\orgdiv{College of Engineering},
  \orgname{Boston University},
  \orgaddress{\city{Boston}, \postcode{02215}, \state{MA}, \country{USA}}}

\affil[2]{\orgdiv{Graduate Summer Fellow of Institute for Global Sustainability},
  \orgname{Boston University},
  \orgaddress{\city{Boston}, \postcode{02215}, \state{MA}, \country{USA}}}

\affil[3]{\orgdiv{Department of Mathematics},
  \orgname{Northeastern University},
  \orgaddress{\city{Boston}, \postcode{02115}, \state{MA}, \country{USA}}}

\affil[4]{\orgdiv{College of Engineering},
  \orgname{The University of Iowa},
  \orgaddress{\city{Iowa City}, \postcode{52242}, \state{IA},\country{USA}}}

\abstract{
We study optimal harvesting in continuous-time, age-structured population models of McKendrick--von Foerster type, and we compare two harvesting mechanisms. In the \emph{rate-control} formulation, harvesting enters the state equation as an additive removal term; in the \emph{effort-control} formulation, harvesting acts multiplicatively as an additional mortality intensity and the mortality coefficient depends on the aggregate stock. For the rate-control problem we first establish existence of an optimal control for the infinite-horizon discounted problem, and we then derive, \emph{under explicitly stated regularity and constraint-qualification assumptions}, a conditional Pontryagin-type necessary optimality system consisting of the adjoint equation, the terminal-age and transversality
conditions, the switching relations for the distributed and boundary controls, and the complementary-slackness relation for the state constraint. For the effort-control problem we \emph{formally} derive the associated adjoint equation and identify the nonlocal coupling term generated by aggregate (density) dependence, with the sign of that term
verified by a step-by-step variational computation; a rigorous infinite-horizon maximum principle for this nonlinear, nonlocally coupled problem is beyond the present scope and is stated as such. We complement the analysis with autonomous stationary reductions, with explicit representations of the state and adjoint, and with a reproducible numerical
study. The central message is structural: the harvesting mechanism is not a cosmetic modelling choice. Rate-control produces an additive/affine/local optimality structure, whereas effort-control produces a multiplicative/nonlinear/nonlocal one, with direct consequences for persistence, stationary profiles, and bioeconomic interpretation.
}


\maketitle

\section{Introduction}

Age-structured population models provide a classical framework for
describing the evolution of heterogeneous populations in which the
demographic state depends explicitly on age. In continuous time, such
dynamics are typically formulated through McKendrick--von Foerster type
transport equations
\citep{mkendrick_applications_1925,von_foerster_remarks_1959,liu2025bidirectional}, which have
become a standard tool in population dynamics, demography, resource
management, and harvesting theory. A systematic mathematical treatment of
these models is developed in the foundational works of Webb
\citep{webb_semigroup_1984,webb_theory_1985,liang2026separation}, Iannelli
\citep{iannelli_mathematical_1995}, Metz et al.\
\citep{metz_dynamics_1986}, and Iannelli and Milner
\citep{iannelli_basic_2017}. Within this framework, deriving optimal
harvesting policies under dynamical constraints is a central question.

In this paper we focus on two harvesting mechanisms within a common
age-structured framework. The first is a \emph{rate-control} formulation,
in which harvesting enters the state equation as a direct removal term. The
second is an \emph{effort-control} formulation, in which harvesting acts as
an additional mortality intensity and the instantaneous yield is
proportional to both effort and stock abundance; the natural mortality
coefficient is moreover allowed to depend on the aggregate stock, which
renders the dynamics nonlinear and nonlocal. For the rate-control problem we
prove existence of an optimiser for the infinite-horizon discounted problem
and then derive necessary optimality conditions of Pontryagin type --- the
adjoint system, the transversality condition, the switching relations for
the distributed harvesting control and the boundary inflow control, and the
complementary-slackness relation for the state constraint --- under
explicitly stated regularity and constraint-qualification assumptions. We
then compare the two formulations and exhibit the nonlocal coupling term
that appears in the effort-control adjoint equation, thereby making explicit
how the choice of harvesting mechanism changes the analytical structure of
the optimality system.

\paragraph{Precise mathematical contribution.}
To position the present work relative to the existing optimal-harvesting
literature, we state our contribution sharply, separating what is new from
what is classical.
\begin{enumerate}[label=(\roman*),leftmargin=2.2em]
\item The individual ingredients of the rate-control optimality system ---
an affine McKendrick transport equation, a backward adjoint equation with
terminal condition $\lambda(t,A)=0$, bang-bang switching laws, and
complementary slackness for the state constraint --- are individually
classical
\citep{brokate_pontryagins_1985,barbu_optimal_1999,anita_analysis_2000,yu2026size}.
Our contribution is therefore \emph{not} a new maximum principle for a single
model, but the explicit side-by-side derivation that isolates how the
harvesting mechanism \emph{alone} alters the structure of the optimality
system.
\item For the rate-control problem we additionally establish existence of an
optimal control for the infinite-horizon discounted problem
(Theorem~\ref{thm:existence_PR}), which anchors the necessary conditions to
an actual optimiser rather than presenting them in the abstract.
\item The nonlocal adjoint term for the effort-control model is, in our
setting, a transparent consequence of density-dependent mortality. It is
consistent with known density-dependent age-structured control results
\citep{boucekkine_optimal_2013} rather than a wholly new phenomenon; we
present it as a \emph{formal} identification accompanied by an explicit,
sign-verified derivation (Section~\ref{sec:effort}). A rigorous
infinite-horizon maximum principle for the effort-control model is not
claimed and is left as an open problem.
\end{enumerate}
The value added is thus the \emph{unified comparison} together with the
precise localisation of the dichotomy
\begin{equation*}
\begin{aligned}
\text{rate-control}\;&\Rightarrow\;\text{additive / affine / local},\\[4pt]
\text{effort-control}\;&\Rightarrow\;\text{multiplicative / nonlinear /
nonlocal},
\end{aligned}
\end{equation*}
the existence result that grounds the rate-control conditions, and a
reproducible numerical comparison of the two mechanisms.

Figure~\ref{fig:mechanism_overview} provides a mechanistic map of the paper.
The paper is organised as follows. Section~\ref{sec:litreview} reviews the
relevant literature, organised around the central comparison.
Section~\ref{sec:model} introduces the two model classes, their real-world
interpretation, and the standing assumptions. Section~\ref{sec:optcontrol}
formulates the infinite-horizon optimal control problem for the rate-control
model, establishes existence of an optimiser, and derives the conditional
necessary optimality system. Section~\ref{sec:effort} treats the
effort-control formulation and the associated formal adjoint equation with
nonlocal dependence. Section~\ref{sec:autonomous} collects autonomous
reductions and explicit stationary formulas and relates them to the optimal
control problem. Sectio~\ref{sec:numerics} presents the numerical study.
Section~\ref{sec:discussion} concludes, separating the proven results from
the formal observations.

\begin{figure}[htbp]
\centering
\includegraphics[width=0.65\linewidth]{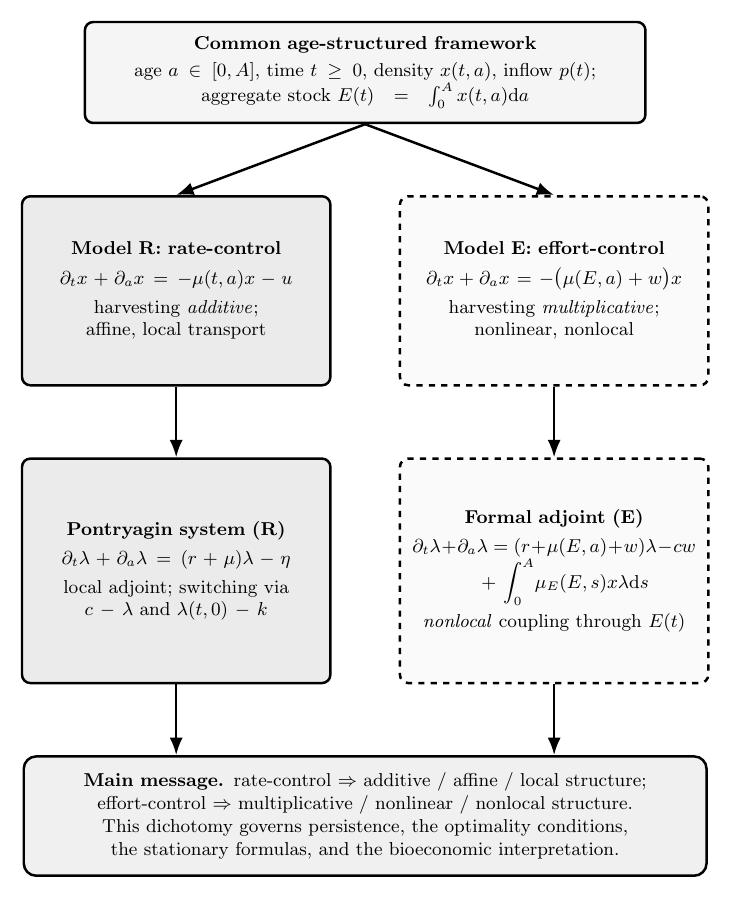}
\caption{Mechanism map of the paper. Starting from a unified age-structured
transport framework, the analysis splits into rate-control (solid) and
effort-control (dashed) harvesting. The former yields a local Pontryagin
system; the latter generates a nonlinear, nonlocal adjoint through aggregate
population dependence.}
\label{fig:mechanism_overview}
\end{figure}

\section{Literature review}
\label{sec:litreview}

The present paper lies at the intersection of five strands, which we review
in turn, organised around the central comparison rather than as a list.

\paragraph{(a) Age-structured population dynamics.}
The continuous-time theory is rooted in the McKendrick--von Foerster
equation \citep{mkendrick_applications_1925,von_foerster_remarks_1959,yu2026pattern} and
has been systematically developed through nonlinear, semigroup, and
abstract-evolution approaches
\citep{webb_semigroup_1984,webb_theory_1985,gao2022rolling,iannelli_mathematical_1995,%
iannelli_basic_2017,metz_dynamics_1986,cushing_introduction_1998,%
engel_one-parameter_2000,banasiak_perturbations_2006,magal_theory_2018},
including recent models with nonlocal feedback and physiological structure
\citep{wang2025analysis,wang2025analysis1,liang2025global}. The associated
steady-state, threshold, and stability theory --- from early nonlinear and
implicit-function frameworks
\citep{gurtin_non-linear_1974,gurtin_simple_1979,cushing_existence_1984,%
pruss_equilibrium_1981,pruss_stability_1983,busenberg_endemic_1988,liu2026ftu,%
busenberg_global_1991,diekmann_steady-state_2003} to general reproduction-number
and pattern formulations
\citep{thieme_mathematics_2018,inaba_age-structured_2006,inaba_stable_2017,%
walker_positive_2009,wang2026algebraic,yu2026beyond,%
wang2026damage,wang2026breakdown} --- underlies the autonomous
reductions of Section~\ref{sec:autonomous}. This body of work provides the
transport framework underlying both of our models.

\paragraph{(b) Optimal harvesting in structured populations.}
Optimal control for age-structured systems has been developed in both
finite- and infinite-horizon settings
\citep{brokate_pontryagins_1985,barbu_optimal_1999,anita_analysis_2000,yu2026rigorous,
anita_optimal_1998,anita_optimal_2009,zhao_optimal_2005,fister_optimal_2004},
with particular attention to the structural properties of optimal policies
\citep{hritonenko_structure_2007,liu2026computational,veliov_optimal_2008} and recent steps toward
stochastic extensions \citep{yu2026microscopic}.
The rate-control optimality system we derive is classical in form; our use
of it is comparative.

\paragraph{(c) Infinite-horizon maximum principles.}
Infinite-horizon problems introduce the well-known difficulty of
transversality conditions at infinity, which has been treated rigorously in
\citep{arrow_optimal_1969,carlson_infinite_1991,aseev_pontryagin_2004,%
aseev_infinite-horizon_2012,yu2026chemotactic,aseev_maximum_2015}. We invoke this literature
when stating, and discussing sufficient conditions for, our transversality
relation.

\paragraph{(d) Rate-control versus effort-control in bioeconomics.}
The distinction between direct harvesting rates and effort-based mortality
is classical in resource economics
\citep{clark_economics_1973,clark_economics_1975,clark_mathematical_2010,yu2026fullcovariance,%
getz_population_1989,reed_optimal_1979,tahvonen_economics_2009}. Mathematically, rate-control preserves an affine
transport structure, whereas effort-control introduces a bilinear coupling,
often complicated by density-dependent nonlocal terms
\citep{boucekkine_optimal_2013,wang2025multi}. 
Recent study established the optimal control results for the effort control model \cite{yu2026size}.
This is precisely the dichotomy our paper
makes explicit at the level of the adjoint system.

\paragraph{(e) Gap addressed by this paper.}
Although the two mechanisms are individually well studied, a direct,
side-by-side derivation that isolates their effect on the
\emph{infinite-horizon optimality system} --- and that locates the emergence
of a nonlocal adjoint term induced by aggregate dependence --- does not
appear to be available in unified form. The present paper fills this gap
within the McKendrick--von Foerster framework.

\section{Model and assumptions}
\label{sec:model}

We consider age-structured population models on the age interval $[0,A]$,
where $A\in(0,\infty)$ is the maximal age and $t\ge0$ is time. The state
variable $x(t,a)\ge0$ is the population density of individuals of age $a$ at
time $t$. The boundary input $p(t)\ge0$ describes the inflow of newborn or
newly introduced individuals at age $a=0$. We compare two harvesting
mechanisms: a rate-control model, in which harvesting acts as a direct
removal term, and an effort-control model, in which harvesting enters as an
additional mortality intensity and may depend on the aggregate population
size. The rate-control model is the main setting for the infinite-horizon
maximum principle; the effort-control model is used to identify the
structural changes induced by multiplicative control and nonlocal population
dependence.

\subsection{Real-world interpretation of the two models}
\label{subsec:interpretation}

Before stating the equations we make the modelling choices concrete, since
the boundary condition and the mortality law are exactly the points at which
the two models depart from the classical renewal model.

\paragraph{The boundary input $p(t)$ as a managed recruitment/stocking
control.}
In the classical Lotka--McKendrick (Kermack--McKendrick) renewal model the
boundary condition is the nonlocal birth law
\begin{equation}\label{eq:renewal_bc}
x(t,0)=\int_0^A \beta(a)\,x(t,a)\dd a,
\end{equation}
where $\beta(a)$ is the age-specific fertility. We deliberately replace
\eqref{eq:renewal_bc} by an \emph{exogenous, controllable} inflow
$x(t,0)=p(t)$. This is not dictated by solvability: the renewal law
\eqref{eq:renewal_bc} is itself well posed. Rather, $p(t)$ models situations
in which recruitment is a \emph{management decision} rather than an internal
demographic outcome --- stocking of a fishery from a hatchery, release of
hatchery-reared or farmed juveniles into a managed population,
re-introduction programmes, or reseeding of a cultivated stand. In such
systems the manager controls how many young individuals enter the population
at $a=0$, and the boundary input is therefore naturally a decision variable
subject to a cost $k(t)$. Treating recruitment as a control (rather than a
fertility integral) is precisely what makes the boundary inflow part of the
optimisation. We note that the framework extends to the renewal boundary law
\eqref{eq:renewal_bc} with $p(t)$ added as a controllable stocking term; we
keep the pure-control form to isolate the effect of the harvesting mechanism.

\paragraph{Density-dependent mortality $\mu(E(t),a)$ in Model E.}
In Model E the mortality coefficient $\mu(E(t),a)$ depends on the aggregate
stock $E(t)=\int_0^A x(t,a)\dd a$ as well as on age. This captures
\emph{crowding and intraspecific competition}: as total abundance grows,
competition for food, space, or other limiting resources raises per-capita
mortality. A simple and widely used instance is the affine law
$\mu(E,a)=\mu_0(a)+\alpha E$ with $\alpha>0$, the resource-limited analogue
of logistic self-limitation; we use exactly this form in the numerical study
(Section~\ref{sec:numerics}). The coupling through $E(t)$ is what makes the
state equation nonlinear and nonlocal. We keep the boundary input in the
simple controllable form $p(t)$ for the same reason as above and in order to
attribute the nonlocal structure unambiguously to the mortality law rather
than to the recruitment law. A representative real-world scenario for Model
E is a stocked, crowding-limited fishery in which natural mortality rises
with total biomass and the instantaneous catch is proportional to fishing
effort and to local abundance (the classical catch-per-unit-effort
assumption).

\paragraph{Notation: $E(t)$ denotes aggregate \emph{stock}, not effort.}
Throughout, $E(t)$ always denotes the aggregate population size (total
stock) defined in \eqref{eq:E_aggregate}; the harvesting \emph{effort} in
Model E is the distributed control $w(t,a)$. We emphasise this to avoid the
common clash of notation in bioeconomics.

\subsection{The two model classes}

\paragraph{Model R (rate-control).}
The age-structured dynamics with direct harvesting rate $u$ is
\begin{align}
\partial_t x(t,a)+\partial_a x(t,a)
&=-\mu(t,a)\,x(t,a)-u(t,a),
&& t>0,\; a\in(0,A), \label{eq:R_state}\\[2pt]
x(t,0)&=p(t), && t\ge 0, \label{eq:R_bc}\\[2pt]
x(0,a)&=x_0(a), && a\in[0,A]. \label{eq:R_ic}
\end{align}
Here $\mu(t,a)\ge0$ is the natural mortality rate, $u(t,a)\ge0$ is the
harvesting rate applied to individuals of age $a$, and $p(t)\ge0$ is the
boundary recruitment/stocking control.

\paragraph{Model E (effort-control).}
The age-structured dynamics with multiplicative harvesting effort $w$ is
\begin{align}
\partial_t x(t,a)+\partial_a x(t,a)
&=-\bigl(\mu(E(t),a)+w(t,a)\bigr)x(t,a), && t>0,\; a\in(0,A),
\label{eq:E_state}\\
E(t)&=\int_0^A x(t,a)\dd a, && t\ge 0, \label{eq:E_aggregate}\\
x(t,0)&=p(t), && t\ge 0, \label{eq:E_bc}\\[2pt]
x(0,a)&=x_0(a), && a\in[0,A]. \label{eq:E_ic}
\end{align}
Here $w(t,a)\ge0$ is the harvesting effort and $\mu(E(t),a)$ is the
density-dependent mortality of Section~\ref{subsec:interpretation}.
Figure~\ref{fig:characteristics} illustrates the evolution along
characteristics in both models.

\begin{figure}[htbp]
    \centering
    \includegraphics[width=\textwidth]{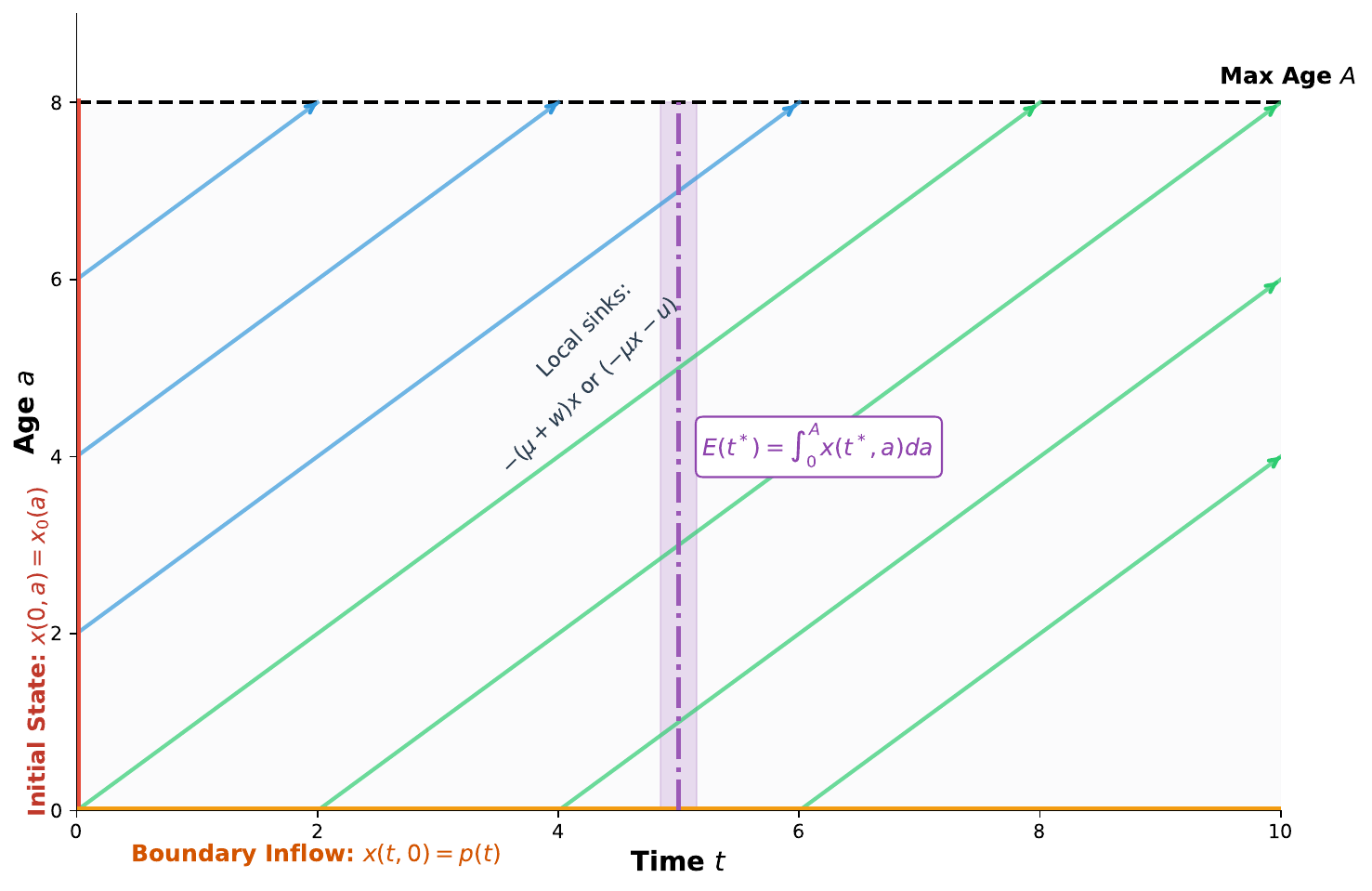}
    \caption{Evolution of the age-structured population along characteristic
    curves in the $(t,a)$ plane. The state $x(t,a)$ propagates along the
    characteristics $\mathrm{d}a/\mathrm{d}t=1$. The boundary inflow control
    $p(t)$ is applied at $a=0$ and the initial profile $x_0(a)$ is given at
    $t=0$. The vertical shaded region illustrates the nonlocal aggregate
    stock $E(t^\ast)$, obtained by integrating the density over all ages at a
    fixed time $t^\ast$.}
    \label{fig:characteristics}
\end{figure}

\subsection{Controls, constraints, and admissible sets}

The control variables are the distributed harvesting control and the
boundary inflow. The constants
\begin{equation}\label{eq:param_defs}
u_{\max},\,w_{\max},\,p_{\max}\in(0,\infty)
\end{equation}
denote, respectively, the maximal admissible harvesting rate (Model R),
maximal admissible harvesting effort (Model E), and maximal admissible
boundary inflow; they are fixed throughout. The remaining data are the
discount rate $r>0$, the unit harvesting value $c(t,a)\ge0$, and the boundary
inflow cost $k(t)\ge0$; all are defined where they first appear in
Section~\ref{sec:optcontrol}.

For Model~R the box constraints are
\begin{equation}
0\le u(t,a)\le u_{\max}, \qquad 0\le p(t)\le p_{\max},
\qquad \text{for a.e. } (t,a), \label{eq:R_controls}
\end{equation}
and for Model~E
\begin{equation}
0\le w(t,a)\le w_{\max}, \qquad 0\le p(t)\le p_{\max},
\qquad \text{for a.e. } (t,a). \label{eq:E_controls}
\end{equation}
In both models the state is required to remain nonnegative,
\begin{equation}
x(t,a)\ge 0 \qquad \text{for a.e. } (t,a)\in[0,\infty)\times[0,A].
\label{eq:state_nonneg}
\end{equation}

We define the \emph{box-admissible} sets
\begin{align}
\mathcal{U}_R
&:=\bigl\{(u,p): u\in L^\infty,\;p\in L^\infty,\;
\eqref{eq:R_controls}\text{ holds}\bigr\}, \label{eq:UR}\\
\mathcal{U}_E
&:=\bigl\{(w,p): w\in L^\infty,\;p\in L^\infty,\;
\eqref{eq:E_controls}\text{ holds}\bigr\}. \label{eq:UE}
\end{align}
Because the additive removal in Model~R can in principle drive $x$ below zero,
the box constraints alone do not guarantee feasibility. We therefore define
the \emph{state-admissible} sets that incorporate \eqref{eq:state_nonneg}
explicitly:
\begin{align}
\mathcal{U}_{R,\mathrm{ad}}
&:=\bigl\{(u,p)\in\mathcal{U}_R:\ \text{the mild solution } x=S(u,p)
\text{ of \eqref{eq:R_state}--\eqref{eq:R_ic} satisfies }
\eqref{eq:state_nonneg}\bigr\}, \label{eq:URad}\\
\mathcal{U}_{E,\mathrm{ad}}
&:=\bigl\{(w,p)\in\mathcal{U}_E:\ \text{the mild solution } x
\text{ of \eqref{eq:E_state}--\eqref{eq:E_ic} satisfies }
\eqref{eq:state_nonneg}\bigr\}. \label{eq:UEad}
\end{align}
The optimisation in Section~\ref{sec:optcontrol} is carried out over
$\mathcal{U}_{R,\mathrm{ad}}$.

\begin{remark}[Nonemptiness and stability of $\mathcal{U}_{R,\mathrm{ad}}$]
\label{rem:URad}
The set $\mathcal{U}_{R,\mathrm{ad}}$ is nonempty: the control $u\equiv0$
(no harvesting) together with any $p\in[0,p_{\max}]$ yields, by the
representation \eqref{eq:explicit_rate_profile} below with $u\equiv0$, a
nonnegative state, so $(0,p)\in\mathcal{U}_{R,\mathrm{ad}}$ whenever
$x_0\ge0$. Moreover $\mathcal{U}_{R,\mathrm{ad}}$ is convex and weak-$\ast$
closed in $L^\infty\times L^\infty$: the solution map $S$ is affine and
weak-$\ast$-to-$L^1_{\mathrm{loc}}$ sequentially continuous (the transport
semigroup is bounded and the removal term is linear), and the constraint
\eqref{eq:state_nonneg} is preserved under $L^1_{\mathrm{loc}}$ limits;
hence the feasible set is the intersection of the weak-$\ast$ compact box
$\mathcal{U}_R$ with a weak-$\ast$ closed convex set. This convexity and
closedness are exactly what the existence proof
(Theorem~\ref{thm:existence_PR}) uses. Stability under admissible
perturbations is local: if $(u,p)\in\mathcal{U}_{R,\mathrm{ad}}$ activates
the constraint only on a set where the removal can be reduced, then small
admissible variations that do not increase $u$ on that set remain feasible;
this is the content of the one-sided switching analysis in
Theorem~\ref{thm:necessary_PR}.
\end{remark}

\subsection{Standing assumptions}

\begin{enumerate}[leftmargin=3em]
    \item[(A1)] $A\in(0,\infty)$.
    \item[(A2)] $x_0\in L^1(0,A)$ and $x_0(a)\ge0$ for a.e.\ $a\in(0,A)$.
    \item[(A3)] In Model~R, $\mu:[0,\infty)\times[0,A]\to[0,\infty)$ is
    bounded and measurable.
    \item[(A4)] In Model~E,
    $\mu:[0,\infty)\times[0,A]\to[0,\infty)$, $(E,a)\mapsto\mu(E,a)$, is
    bounded on bounded sets and measurable in $a$. When variational
    arguments are needed for the formal adjoint calculation we additionally
    assume that $\mu(\cdot,a)$ is continuously differentiable in $E$ for
    a.e.\ $a$, with derivative $\mu_E$ bounded on bounded sets.
    \item[(A5)] The controls belong to $\mathcal{U}_R$ or $\mathcal{U}_E$ as
    appropriate, and $p\ge0$.
    \item[(A6)] \emph{(Admissibility)} Only state-admissible control
    pairs are considered, i.e.\ $(u,p)\in\mathcal{U}_{R,\mathrm{ad}}$ for
    Model~R and $(w,p)\in\mathcal{U}_{E,\mathrm{ad}}$ for Model~E; equivalently,
    the mild solution satisfies the nonnegativity constraint
    \eqref{eq:state_nonneg}.
\end{enumerate}
Under these assumptions each model is understood in the usual mild sense
along characteristics. For Model~R this is a linear transport problem; for
Model~E the aggregate term $E(t)$ and the multiplicative control make the
problem nonlinear. These two features are precisely what distinguish the
two harvesting mechanisms.

For ease of reference we summarise the symbols:
\[
t:\text{ time},\quad a:\text{ age},\quad x(t,a):\text{ age density},\quad
p(t):\text{ boundary inflow (control)},
\]
\[
u(t,a):\text{ harvesting rate (Model R)},\quad
w(t,a):\text{ harvesting effort (Model E)},
\]
\[
E(t)=\int_0^A x(t,a)\dd a:\text{ aggregate stock}.
\]

\section{Infinite-horizon optimal control for the rate-control model}
\label{sec:optcontrol}

We consider the infinite-horizon discounted harvesting problem for the
rate-control model. We first record the optimisation problem and prove
existence of an optimiser; we then derive, under explicitly stated
regularity and constraint-qualification assumptions, a conditional
Pontryagin-type necessary optimality system.

\begin{problem}\label{prob:PR}
Maximise
\begin{equation}\label{eq:PR_obj}
J(u,p)=\int_0^\infty e^{-rt}
\left(\int_0^A c(t,a)\,u(t,a)\dd a-k(t)\,p(t)\right)\dd t
\end{equation}
over all state-admissible control pairs $(u,p)\in\mathcal{U}_{R,\mathrm{ad}}$,
where for each such pair the state $x=S(u,p)$ is the unique mild solution of
\begin{align}
\partial_t x+\partial_a x&=-\mu(t,a)x-u(t,a), && t>0,\ a\in(0,A),
\label{eq:PR_state}\\
x(t,0)&=p(t), && t\ge0, \label{eq:PR_bc}\\
x(0,a)&=x_0(a), && a\in[0,A]. \label{eq:PR_ic}
\end{align}
Here $r>0$ is the discount rate, $c(t,a)\ge0$ is the unit harvesting value,
and $k(t)\ge0$ is the boundary inflow cost.
\end{problem}

\begin{remark}[What is maximised]
The objective \eqref{eq:PR_obj} depends only on the control pair $(u,p)$:
the state $x$ does not appear in $J$. The maximisation is therefore over
controls $(u,p)\in\mathcal{U}_{R,\mathrm{ad}}$, with the state determined
through the solution map $x=S(u,p)$ and entering only through the
feasibility requirement $x\ge0$ encoded in $\mathcal{U}_{R,\mathrm{ad}}$.
Equivalently, one maximises over triples $(u,p,x)$ satisfying the state
equation and the constraints.
\end{remark}

\subsection{Existence of an optimal control}

We address existence before turning to necessary
conditions, so that the optimality system is anchored to an actual
optimiser. The structural feature that makes existence elementary here is
that the rate-control objective \eqref{eq:PR_obj} is \emph{linear} in the
control pair.

\begin{theorem}[Existence for the rate-control problem]
\label{thm:existence_PR}
Assume \textup{(A1)--(A3)}, $r>0$, $c\in L^\infty((0,\infty)\times(0,A))$,
$k\in L^\infty(0,\infty)$, and that the control-to-state map
$S:\mathcal{U}_R\to L^1_{\mathrm{loc}}([0,\infty);L^1(0,A))$ is affine and
sequentially weak-$\ast$-to-$L^1_{\mathrm{loc}}$ continuous. If
$\mathcal{U}_{R,\mathrm{ad}}\neq\varnothing$, then Problem~\ref{prob:PR}
admits at least one optimal control
$(u^\ast,p^\ast)\in\mathcal{U}_{R,\mathrm{ad}}$.
\end{theorem}

\begin{proof}
The box $\mathcal{U}_R$ is bounded, convex, and weak-$\ast$ closed in
$L^\infty((0,\infty)\times(0,A))\times L^\infty(0,\infty)$, hence weak-$\ast$
compact by the Banach--Alaoglu theorem (both factors are duals of separable
$L^1$ spaces, so the weak-$\ast$ topology is metrisable on bounded sets and
sequential arguments suffice). By Remark~\ref{rem:URad} the state-admissible
set $\mathcal{U}_{R,\mathrm{ad}}$ is a weak-$\ast$ closed convex subset of
$\mathcal{U}_R$, hence itself weak-$\ast$ compact, and it is nonempty by
hypothesis.

The objective is weak-$\ast$ continuous. Indeed, since $r>0$,
\[
\|e^{-rt}c\|_{L^1((0,\infty)\times(0,A))}
\le \frac{A}{r}\,\|c\|_{L^\infty}<\infty,
\qquad
\|e^{-rt}k\|_{L^1(0,\infty)}\le\frac{1}{r}\,\|k\|_{L^\infty}<\infty,
\]
so $e^{-rt}c$ and $e^{-rt}k$ are admissible $L^1$ test functions, and
\[
J(u,p)=\langle e^{-rt}c,\,u\rangle-\langle e^{-rt}k,\,p\rangle
\]
is, by definition of the weak-$\ast$ topology, continuous along weak-$\ast$
convergent sequences in $\mathcal{U}_R$.

Let $(u_n,p_n)\subset\mathcal{U}_{R,\mathrm{ad}}$ be a maximising sequence,
$J(u_n,p_n)\to\sup J$. By weak-$\ast$ compactness a subsequence converges
weak-$\ast$ to some $(u^\ast,p^\ast)\in\mathcal{U}_{R,\mathrm{ad}}$. By
weak-$\ast$ continuity $J(u^\ast,p^\ast)=\lim J(u_n,p_n)=\sup J$. Hence the
supremum is attained and $(u^\ast,p^\ast)$ is optimal. The discount factor
$r>0$ is what guarantees that $J$ is finite on $\mathcal{U}_R$ and is
weak-$\ast$ continuous on the infinite horizon; without it the objective
need not be well defined.
\end{proof}

\begin{remark}[On Model~E]
The analogous existence question for the effort-control model is genuinely
harder: there the objective $\int_0^\infty e^{-rt}\int_0^A c\,w\,x\dd a\dd t$
is \emph{bilinear} in $(w,x)$ and $x$ depends nonlinearly and nonlocally on
$w$ through $E$. Lower semicontinuity then requires compactness of the state
map and care with the product $w\,x$. We do not pursue existence for Model~E
here; consistently with Section~\ref{sec:effort}, our treatment of Model~E
is formal.
\end{remark}

\subsection{Standing regularity and constraint-qualification assumptions}

The necessary conditions below are obtained from a first variation of a
current-value Lagrangian. To make precise the sense in which they hold, we
collect the assumptions under which the variational calculation is justified.
We stress that several of these assumptions
\emph{postulate} analytical facts (differentiability of the control-to-state
map, regularity of the adjoint and multiplier, a constraint qualification)
that a fully rigorous infinite-horizon PDE-constrained theory would have to
\emph{establish}. Accordingly, Theorem~\ref{thm:necessary_PR} is stated and
should be read as a \emph{conditional} first-order optimality system: a
necessary system \emph{under} (R1)--(R7), not an unconditional maximum
principle.

\begin{enumerate}[leftmargin=3.4em]
    \item[\textup{(R1)}] \textbf{Data and admissible controls.}
    $c,\mu\in L^\infty((0,\infty)\times(0,A))$, $k\in L^\infty(0,\infty)$,
    $r>0$; $x_0\in L^1(0,A)$, $x_0\ge0$ a.e.; and the controls satisfy the
    box constraints \eqref{eq:R_controls}.
    \item[\textup{(R2)}] \textbf{Well-posedness.} For every
    $(u,p)\in\mathcal{U}_R$ the state equation
    \eqref{eq:PR_state}--\eqref{eq:PR_ic} has a unique mild solution
    $x=S(u,p)\in L^1_{\mathrm{loc}}([0,\infty);L^1(0,A))$.
    \item[\textup{(R3)}] \textbf{Differentiability of $S$.} The map $S$ is
    G\^ateaux differentiable along admissible directions, and for every
    admissible variation $(\delta u,\delta p)$ the derivative
    $\delta x=DS(u,p)[\delta u,\delta p]$ is the unique mild solution of the
    linearised system \eqref{eq:pf_lin_state}--\eqref{eq:pf_lin_bc}.
    \item[\textup{(R4)}] \textbf{Integrability of the objective.} For every
    $(u,p)\in\mathcal{U}_R$,
    $\int_0^\infty e^{-rt}(\int_0^A|c\,u|\dd a+|k\,p|)\dd t<\infty$.
    \item[\textup{(R5)}] \textbf{Regularity of adjoint and multiplier.} The
    adjoint $\lambda$ and the state-constraint multiplier $\eta$ belong to
    classes for which all terms in the Lagrangian and the
    integration-by-parts formulas are meaningful, with
    $e^{-rt}\lambda\in W^{1,1}_{\mathrm{loc}}((0,\infty)\times(0,A))$ and
    $\eta\in L^1_{\mathrm{loc}}((0,\infty)\times(0,A))$, $\eta\ge0$. (See
    Remark~\ref{rem:measure_multiplier} on the measure-valued case.)
    \item[\textup{(R6)}] \textbf{Integration by parts and transversality.}
    The products $e^{-rt}\lambda\,\delta x$, $e^{-rt}\lambda\,\partial_t\delta
    x$, $e^{-rt}\lambda\,\partial_a\delta x$ are integrable on finite
    cylinders and the integration-by-parts identities hold; moreover the
    boundary pairing at infinity vanishes (see
    Lemma~\ref{lem:transversality_suff} for sufficient conditions).
    \item[\textup{(R7)}] \textbf{Constraint qualification.} The pointwise
    state constraint $x\ge0$ is qualified so that there exists a multiplier
    $\eta\ge0$ with complementary slackness $\eta\,x^\ast=0$ a.e.
\end{enumerate}

\begin{remark}[Measure-valued multipliers]
\label{rem:measure_multiplier}
For pointwise pure-state constraints such as $x(t,a)\ge0$, the natural
Lagrange multiplier is in general a nonnegative Radon measure
$\eta\in\mathcal{M}_+([0,\infty)\times[0,A])$ supported on the active set
$\{x^\ast=0\}$, rather than an $L^1_{\mathrm{loc}}$ density; this is the
classical phenomenon for state-constrained control. Assuming
$\eta\in L^1_{\mathrm{loc}}$ in (R5) is therefore a \emph{simplifying
regularity hypothesis}, valid when the active set is ``fat'' enough that the
multiplier admits a density, and convenient because it keeps all terms in
the integration-by-parts formula classical. The entire derivation below goes
through verbatim with $\eta$ a measure provided the products
$e^{-rt}\lambda\dd\eta$ and the pairing $\langle\eta,\delta x\rangle$ are
interpreted in the measure sense and $\delta x$ is taken continuous in
$(t,a)$; the adjoint equation \eqref{eq:adjoint_theorem_PR} then holds in the
sense of measures, with $\eta$ contributing a singular source on the active
set. We record the $L^1$ form for readability and flag the measure-valued
generalisation as the rigorous formulation. This point is especially
relevant for Model~R, where additive removal can activate the constraint on
specific age classes.
\end{remark}

\begin{lemma}[Sufficient conditions for transversality]
\label{lem:transversality_suff}
Suppose there exist constants $C_\lambda,C_x\ge0$ and rates
$\rho,\sigma\ge0$ with $\rho+\sigma<r$ such that, for all large $T$,
\[
\operatorname*{ess\,sup}_{a\in[0,A]}|\lambda(T,a)|\le C_\lambda\,e^{\rho T},
\qquad
\|\delta x(T,\cdot)\|_{L^1(0,A)}\le C_x\,e^{\sigma T}.
\]
Then
$\displaystyle\lim_{T\to\infty}\int_0^A e^{-rT}\lambda(T,a)\delta x(T,a)\dd a=0.$
\end{lemma}
\begin{proof}
By H\"older's inequality,
\[
\left|\int_0^A e^{-rT}\lambda(T,a)\delta x(T,a)\dd a\right|
\le e^{-rT}\,\|\lambda(T,\cdot)\|_{L^\infty}\,\|\delta x(T,\cdot)\|_{L^1}
\le C_\lambda C_x\,e^{-(r-\rho-\sigma)T}\xrightarrow[T\to\infty]{}0,
\]
since $r-\rho-\sigma>0$.
\end{proof}

\begin{remark}[When the hypotheses of Lemma~\ref{lem:transversality_suff} hold]
Under bounded mortality (A3), the linearised state $\delta x$ propagates
along characteristics with at most exponential growth, so
$\|\delta x(T,\cdot)\|_{L^1}\le C_x e^{\sigma T}$ with $\sigma$ controlled by
$\|\mu\|_\infty$ and the size of the admissible variation; and the
discounted adjoint $e^{-rt}\lambda$ is bounded whenever the value profile $c$
and the multiplier $\eta$ are bounded, giving $\rho=0$. The positive discount
rate $r>0$, together with $\sigma<r$, is precisely what closes the estimate.
This is the age-structured counterpart of the standard transversality
conditions for infinite-horizon problems
\citep{aseev_infinite-horizon_2012,aseev_maximum_2015,cai2026optimal}.
\end{remark}

\subsection{The first variation and the optimality system}

We introduce a current-value adjoint $\lambda(t,a)$ for the state equation
and a current-value multiplier $\eta(t,a)\ge0$ for the constraint $x\ge0$,
and form the current-value Lagrangian
\begin{align}
\mathcal L(u,p,x;\lambda,\eta)
&=\int_0^\infty e^{-rt}\left(\int_0^A c\,u\dd a-k\,p\right)\dd t
\notag\\
&\quad-\int_0^\infty\!\!\int_0^A e^{-rt}\lambda
\bigl(\partial_t x+\partial_a x+\mu x+u\bigr)\dd a\dd t
+\int_0^\infty\!\!\int_0^A e^{-rt}\eta\,x\dd a\dd t.
\label{eq:Lagrangian_PR}
\end{align}
Let $(u^\ast,p^\ast,x^\ast)$ be optimal (such a pair exists by
Theorem~\ref{thm:existence_PR}). For admissible $(\delta u,\delta p)$ set
$u^\varepsilon=u^\ast+\varepsilon\delta u$,
$p^\varepsilon=p^\ast+\varepsilon\delta p$ with $|\varepsilon|$ small enough
to preserve feasibility, and let
$\delta x=\frac{d}{d\varepsilon}x^\varepsilon|_{\varepsilon=0}$. Then
\begin{align}
\partial_t \delta x+\partial_a \delta x
&=-\mu\,\delta x-\delta u, \label{eq:pf_lin_state}\\
\delta x(t,0)&=\delta p(t),\qquad \delta x(0,a)=0. \label{eq:pf_lin_bc}
\end{align}
Differentiating \eqref{eq:Lagrangian_PR} at $\varepsilon=0$, collecting the
$\delta u$-terms, and integrating by parts in $t$ and $a$ to transfer the
derivatives onto $\lambda$ (using $\delta x(0,a)=0$ and
$\delta x(t,0)=\delta p$), one obtains
\begin{align}
\delta\mathcal L
&=\int_0^\infty\!\!\int_0^A e^{-rt}\bigl(c-\lambda\bigr)\delta u\dd a\dd t
+\int_0^\infty e^{-rt}\bigl(\lambda(t,0)-k\bigr)\delta p\dd t
\notag\\
&\quad+\int_0^\infty\!\!\int_0^A e^{-rt}
\Bigl(\partial_t\lambda+\partial_a\lambda-(r+\mu)\lambda+\eta\Bigr)
\delta x\dd a\dd t
\notag\\
&\quad-\int_0^\infty e^{-rt}\lambda(t,A)\delta x(t,A)\dd t
-\lim_{T\to\infty}\int_0^A e^{-rT}\lambda(T,a)\delta x(T,a)\dd a.
\label{eq:pf_deltaL_3}
\end{align}
The interior $\delta x$-integral vanishes for arbitrary $\delta x$ iff the
adjoint satisfies $\partial_t\lambda+\partial_a\lambda=(r+\mu)\lambda-\eta$.
The boundary term at $a=A$ vanishes for arbitrary $\delta x(t,A)$ iff
$\lambda(t,A)=0$ (this terminal-age condition is
forced by the requirement that the $a=A$ boundary pairing, which carries the
\emph{outflow} of first variations at the maximal age, contribute nothing to
$\delta\mathcal L$; since $\delta x(t,A)$ is otherwise free, the coefficient
$\lambda(t,A)$ must be zero).
The term at infinity vanishes under the
transversality condition (Lemma~\ref{lem:transversality_suff}). With these
choices,
\begin{equation}\label{eq:pf_deltaL_reduced}
\delta\mathcal L=\int_0^\infty\!\!\int_0^A e^{-rt}(c-\lambda)\delta u\dd a\dd t
+\int_0^\infty e^{-rt}(\lambda(t,0)-k)\delta p\dd t,
\end{equation}
which identifies the switching gradients
$\frac{\partial I}{\partial u}=e^{-rt}(c-\lambda)$ and
$\frac{\partial I}{\partial p}=e^{-rt}(\lambda(t,0)-k)$.

\begin{theorem}[Conditional necessary optimality system]
\label{thm:necessary_PR}
Let $(u^\ast,p^\ast,x^\ast)$ be an optimal solution of
Problem~\ref{prob:PR}, and assume \textup{(R1)--(R7)}. Then there exist a
current-value adjoint $\lambda:[0,\infty)\times[0,A]\to\R$ and a current-value
multiplier $\eta:[0,\infty)\times[0,A]\to[0,\infty)$ (interpreted as a
measure in the sense of Remark~\ref{rem:measure_multiplier} if (R5) is
relaxed) such that:
\begin{enumerate}[leftmargin=2.6em]
    \item[\textup{(i)}] \textbf{Adjoint equation.}
    \begin{equation}\label{eq:adjoint_theorem_PR}
    \partial_t\lambda+\partial_a\lambda=(r+\mu)\lambda-\eta,
    \qquad t>0,\ a\in(0,A).
    \end{equation}
    \item[\textup{(ii)}] \textbf{Terminal-age and transversality conditions.}
    \begin{equation}\label{eq:adjoint_terminal_theorem_PR}
    \lambda(t,A)=0\ (t\ge0),\qquad
    \lim_{T\to\infty}\int_0^A e^{-rT}\lambda(T,a)\delta x(T,a)\dd a=0
    \end{equation}
    for every admissible first variation $\delta x$.
    \item[\textup{(iii)}] \textbf{Switching for the distributed control $u$.}
    With $\frac{\partial I}{\partial u}(t,a):=e^{-rt}(c(t,a)-\lambda(t,a))$,
    \begin{equation}\label{eq:u_switch}
    \frac{\partial I}{\partial u}\le0\ \text{where }u^\ast=0,\quad
    \frac{\partial I}{\partial u}=0\ \text{where }0<u^\ast<u_{\max},\quad
    \frac{\partial I}{\partial u}\ge0\ \text{where }u^\ast=u_{\max}.
    \end{equation}
    \item[\textup{(iv)}] \textbf{Switching for the boundary control $p$.}
    With $\frac{\partial I}{\partial p}(t):=e^{-rt}(\lambda(t,0)-k(t))$,
    \begin{equation}\label{eq:p_switch}
    \frac{\partial I}{\partial p}\le0\ \text{where }p^\ast=0,\quad
    \frac{\partial I}{\partial p}=0\ \text{where }0<p^\ast<p_{\max},\quad
    \frac{\partial I}{\partial p}\ge0\ \text{where }p^\ast=p_{\max}.
    \end{equation}
    \item[\textup{(v)}] \textbf{Complementary slackness.}
    $\eta\ge0$, $x^\ast\ge0$, and $\eta\,x^\ast=0$ a.e.\ (equivalently,
    $\operatorname{supp}\eta\subseteq\{x^\ast=0\}$ in the measure-valued
    formulation).
\end{enumerate}
Moreover, given the multiplier $\eta$ and the data, the adjoint $\lambda$ is
\emph{unique}: it is the unique solution of the backward transport problem
\eqref{eq:adjoint_theorem_PR} with terminal condition
$\lambda(t,A)=0$, integrated along characteristics from $a=A$ downward
(cf.\ the explicit stationary formula \eqref{eq:auto_adjoint_formula}). The
multiplier $\eta$ is in turn determined on the active set by complementary
slackness.
\end{theorem}

\begin{proof}
Existence of an optimal pair is Theorem~\ref{thm:existence_PR}. By (R3),
(R5), (R6) the first variation of the current-value Lagrangian is
\eqref{eq:pf_deltaL_3}, and choosing $\lambda$, the terminal-age condition,
and the transversality condition as above reduces it to
\eqref{eq:pf_deltaL_reduced}. This proves (i), (ii).

For (iii)--(iv) we use the normal-cone characterisation of the box
constraints. At the optimum,
$\delta\mathcal L\le0$ for every feasible direction $(\delta u,\delta p)$,
i.e.\ for every $(\delta u,\delta p)$ in the tangent cone to the box
$[0,u_{\max}]\times[0,p_{\max}]$ at $(u^\ast,p^\ast)$. Equivalently, the
gradient pair $(\frac{\partial I}{\partial u},\frac{\partial I}{\partial p})$
lies in the normal cone $N_{\mathcal U_R}(u^\ast,p^\ast)$. For the pointwise
box constraint $0\le u^\ast(t,a)\le u_{\max}$ this normal cone decomposes
pointwise as
\[
N(t,a)=\begin{cases}
(-\infty,0], & u^\ast(t,a)=0,\\
\{0\}, & 0<u^\ast(t,a)<u_{\max},\\
[0,\infty), & u^\ast(t,a)=u_{\max},
\end{cases}
\]
which is exactly \eqref{eq:u_switch}; the same argument applied to
$p^\ast$ gives \eqref{eq:p_switch}. Concretely: where $0<u^\ast<u_{\max}$ both
signs of $\delta u$ are admissible, forcing
$\frac{\partial I}{\partial u}=0$; where $u^\ast=0$ only $\delta u\ge0$ is
admissible, forcing $\frac{\partial I}{\partial u}\le0$; where
$u^\ast=u_{\max}$ only $\delta u\le0$ is admissible, forcing
$\frac{\partial I}{\partial u}\ge0$.

Part (v) is the constraint qualification (R7): the pointwise state constraint
$x^\ast\ge0$ is enforced by a multiplier $\eta\ge0$ satisfying complementary
slackness $\eta\,x^\ast=0$. Uniqueness of $\lambda$ follows because
\eqref{eq:adjoint_theorem_PR} with $\lambda(t,A)=0$ is a backward transport
equation with a one-sided (terminal-age) condition, integrable uniquely along
characteristics.
\end{proof}

\begin{remark}[Status of the theorem]
Theorem~\ref{thm:necessary_PR} is a \emph{conditional} optimality system:
under (R1)--(R7) it is a set of necessary conditions for the optimiser whose
existence is guaranteed by Theorem~\ref{thm:existence_PR}. We do not claim to
have established the functional-analytic foundations (G\^ateaux
differentiability of $S$, the precise multiplier space, the constraint
qualification) from the data; these are hypothesised in (R3), (R5), (R7). A
fully unconditional infinite-horizon maximum principle for this
state-constrained transport problem --- in particular the rigorous
construction of the transversality relation in the spirit of
\citep{aseev_pontryagin_2004,aseev_maximum_2015,wang2026elliptic} and the measure-valued
multiplier of Remark~\ref{rem:measure_multiplier} --- is left open. The
purpose here is to identify the structural form of the optimality system for
comparison with the effort-control case.
\end{remark}

\subsection{Using the optimality system to compute controls}
\label{subsec:numerical_use_R}

We indicate how Theorem~\ref{thm:necessary_PR} is used to compute (or
approximate) an optimal control. Because the
current-value Hamiltonian is \emph{linear} in $(u,p)$, the optimal control is
determined by the sign of the switching gradients and is generically
\emph{bang-bang}:
\begin{equation}\label{eq:bangbang_R}
u^\ast(t,a)=
\begin{cases}
u_{\max}, & c(t,a)-\lambda(t,a)>0,\\
0, & c(t,a)-\lambda(t,a)<0,
\end{cases}
\qquad
p^\ast(t)=
\begin{cases}
p_{\max}, & \lambda(t,0)-k(t)>0,\\
0, & \lambda(t,0)-k(t)<0,
\end{cases}
\end{equation}
with singular arcs only where a switching function vanishes on a set of
positive measure. This yields a constructive \emph{forward--backward sweep}
scheme: (i) fix a control $(u,p)$ and solve the forward state equation
\eqref{eq:PR_state}--\eqref{eq:PR_ic} along characteristics; (ii) solve the
backward adjoint \eqref{eq:adjoint_theorem_PR} with $\lambda(t,A)=0$,
updating the multiplier $\eta$ on the active set $\{x=0\}$ by complementary
slackness (e.g.\ by a projection/penalisation that keeps $x\ge0$); (iii)
update the control from the switching law \eqref{eq:bangbang_R} (or, on
singular arcs, by gradient projection onto the box using
$\frac{\partial I}{\partial u}$, $\frac{\partial I}{\partial p}$); (iv)
iterate until the switching conditions \eqref{eq:u_switch}--\eqref{eq:p_switch}
hold to a prescribed tolerance. The state constraint is handled numerically
by the truncation/projection described in Section~\ref{sec:numerics}. The
illustrative computation in Section~\ref{subsec:adjoint_numeric} carries out
the adjoint solve and switching step \eqref{eq:bangbang_R} for a prescribed
multiplier, producing what we there call a \emph{bang-bang proxy} control.

\section{Effort-control model and the nonlocal adjoint term (formal)}
\label{sec:effort}

We now turn to the effort-control formulation. In contrast to the
rate-control model, the state equation is nonlinear because the mortality
coefficient depends on the aggregate stock $E(t)=\int_0^A x(t,a)\dd a$, and
the control enters multiplicatively through $w(t,a)x(t,a)$. 
We treat this model at the level of a \emph{formal} variational derivation under sufficient smoothness; we make this status explicit and do not claim a maximum principle.

\begin{quote}
\textbf{Scope statement.} A rigorous infinite-horizon maximum principle for
Model~E --- including existence of an optimiser, the precise functional
setting for the nonlinear nonlocal state map, and the rigorous transversality
analysis --- is beyond the scope of the present paper. The result below
(Proposition~\ref{prop:formal_effort_adjoint}) is a \emph{formal}
identification of the adjoint equation under the smoothness hypotheses stated
in its premises.
\end{quote}

For the effort-control problem consider the infinite-horizon discounted
functional
\begin{equation}\label{eq:PE_obj}
J_E(w,p,x)=\int_0^\infty e^{-rt}
\left(\int_0^A c(t,a)\,w(t,a)x(t,a)\dd a-k(t)\,p(t)\right)\dd t,
\end{equation}
subject to \eqref{eq:E_state}--\eqref{eq:E_ic}. Introduce the current-value
adjoint $\lambda(t,a)$ and the current-value Lagrangian
\begin{align}
\mathcal L_E(w,p,x;\lambda)
&=\int_0^\infty e^{-rt}\left(\int_0^A c\,w\,x\dd a-k\,p\right)\dd t
\notag\\
&\quad-\int_0^\infty\!\!\int_0^A e^{-rt}\lambda
\Bigl(\partial_t x+\partial_a x+\bigl(\mu(E,a)+w\bigr)x\Bigr)\dd a\dd t.
\label{eq:LE}
\end{align}

\subsection{Step-by-step first variation and the sign of the nonlocal term}
\label{subsec:sign_check}

We compute $\delta\mathcal L_E$ in the direction $\delta x$ induced by an
admissible control variation, displaying each contribution so that the sign
of the nonlocal term is unambiguous. We use the sign
convention of \eqref{eq:LE}: the state equation enters with the multiplier
$-e^{-rt}\lambda$ acting on $\partial_t x+\partial_a x+(\mu(E,a)+w)x$.

\smallskip
\noindent\textbf{(1) Running payoff.}
$\delta_{x}\!\bigl(c\,w\,x\bigr)=c\,w\,\delta x$, contributing
\[
\int_0^\infty\!\!\int_0^A e^{-rt}\,c(t,a)w(t,a)\,\delta x\dd a\dd t,
\qquad\text{coefficient of }\delta x:\ +e^{-rt}cw.
\]

\noindent\textbf{(2) Transport terms.} Integration by parts in $t$ and $a$
(with $\delta x(0,a)=0$, $\delta x(t,0)=\delta p$, and the terminal-age and
infinite-horizon boundary terms treated as in
Section~\ref{sec:optcontrol}) transfers the derivatives onto $\lambda$ and
produces the interior coefficient
\[
\partial_t(e^{-rt}\lambda)+\partial_a(e^{-rt}\lambda)
=e^{-rt}\bigl(\partial_t\lambda+\partial_a\lambda-r\lambda\bigr),
\qquad\text{coefficient of }\delta x:\ +e^{-rt}(\partial_t\lambda+\partial_a\lambda-r\lambda).
\]

\noindent\textbf{(3) Local mortality and multiplicative control.}
$\delta_x\!\bigl[(\mu(E,a)+w)x\bigr]$ contains the local part
$(\mu(E,a)+w)\delta x$, which enters \eqref{eq:LE} with the factor
$-e^{-rt}\lambda$:
\[
\text{coefficient of }\delta x:\ -e^{-rt}\bigl(\mu(E,a)+w\bigr)\lambda.
\]

\noindent\textbf{(4) Nonlocal (aggregate) part.}
The dependence of $\mu$ on $E$ gives the extra term
\[
\delta_x\!\bigl[\mu(E,a)x\bigr]\Big|_{\text{nonlocal}}
=\mu_E(E,a)\,x(t,a)\,\delta E(t),
\qquad
\delta E(t)=\int_0^A\delta x(t,s)\dd s.
\]
Inserting this into $-\int\!\!\int e^{-rt}\lambda\,(\cdots)$ and exchanging the
order of integration,
\begin{align}
&-\int_0^\infty\!\!\int_0^A e^{-rt}\lambda(t,a)\,\mu_E(E,a)\,x(t,a)
\left(\int_0^A\delta x(t,s)\dd s\right)\dd a\dd t
\notag\\
&\qquad=-\int_0^\infty\!\!\int_0^A e^{-rt}
\left(\int_0^A\mu_E(E,\sigma)\,x(t,\sigma)\,\lambda(t,\sigma)\dd\sigma\right)
\delta x(t,s)\dd s\dd t,
\label{eq:nonlocal_variation}
\end{align}
so the coefficient of $\delta x(t,s)$ from the nonlocal part is
\[
-e^{-rt}\int_0^A \mu_E(E,\sigma)\,x(t,\sigma)\,\lambda(t,\sigma)\dd\sigma
\qquad(\text{independent of }s).
\]

\smallskip
\noindent\textbf{Collecting (1)--(4)} and requiring the total coefficient of
$\delta x$ to vanish for all admissible $\delta x$:
\[
e^{-rt}\Bigl[
\partial_t\lambda+\partial_a\lambda-r\lambda
+cw-(\mu(E,a)+w)\lambda
-\int_0^A\mu_E(E,\sigma)x(t,\sigma)\lambda(t,\sigma)\dd\sigma
\Bigr]=0,
\]
hence, solving for $\partial_t\lambda+\partial_a\lambda$,
\begin{equation}\label{eq:sign_checked_adjoint}
\partial_t\lambda+\partial_a\lambda
=\bigl(r+\mu(E,a)+w\bigr)\lambda-c\,w
+\int_0^A\mu_E(E,s)\,x(t,s)\,\lambda(t,s)\dd s.
\end{equation}
The nonlocal integral therefore enters with a $\boldsymbol{+}$ sign, the
running-payoff term with a $\boldsymbol{-}$ sign. The positive sign is a
direct consequence of (a) the $-e^{-rt}\lambda$ convention in \eqref{eq:LE}
and (b) the second minus sign produced when the term is moved to the
right-hand side; the two cancel. We have verified this independently by
re-deriving \eqref{eq:sign_checked_adjoint} from the Hamiltonian
$\mathcal H=e^{-rt}(cwx)-e^{-rt}\lambda(\mu(E,a)+w)x$ and the costate rule
$\partial_t\lambda+\partial_a\lambda=r\lambda-\partial_x\mathcal H/e^{-rt}$
with the G\^ateaux derivative of the nonlocal functional
$x\mapsto\int\mu(E,a)x\,\lambda$; both routes give the same positive sign.

\begin{proposition}[Formal adjoint system for the effort-control problem]
\label{prop:formal_effort_adjoint}
Assume the coefficients, controls, and state are sufficiently smooth, that
$\mu(E,a)$ is continuously differentiable in $E$ with $\mu_E$ bounded on
bounded sets, and that all integration-by-parts steps are justified. Then the
current-value adjoint associated with \eqref{eq:PE_obj}--\eqref{eq:E_ic}
formally satisfies \eqref{eq:sign_checked_adjoint}, i.e.
\begin{align}
\partial_t\lambda+\partial_a\lambda
&=\bigl(r+\mu(E(t),a)+w(t,a)\bigr)\lambda(t,a)-c(t,a)w(t,a)
\notag\\
&\quad+\int_0^A\mu_E(E(t),s)\,x(t,s)\,\lambda(t,s)\dd s,
\qquad t>0,\ a\in(0,A),
\label{eq:formal_effort_adjoint}
\end{align}
together with the formal terminal-age and transversality conditions
\begin{equation}\label{eq:formal_effort_boundary}
\lambda(t,A)=0\ (t\ge0),\qquad
\lim_{T\to\infty}\int_0^A e^{-rT}\lambda(T,a)\delta x(T,a)\dd a=0
\end{equation}
for every admissible first variation $\delta x$.
\end{proposition}

Figure~\ref{fig:architecture} visualizes the infinite-horizon necessary optimality conditions.

\begin{figure}[htbp]
 \centering
 \includegraphics[width=1.3\textwidth]{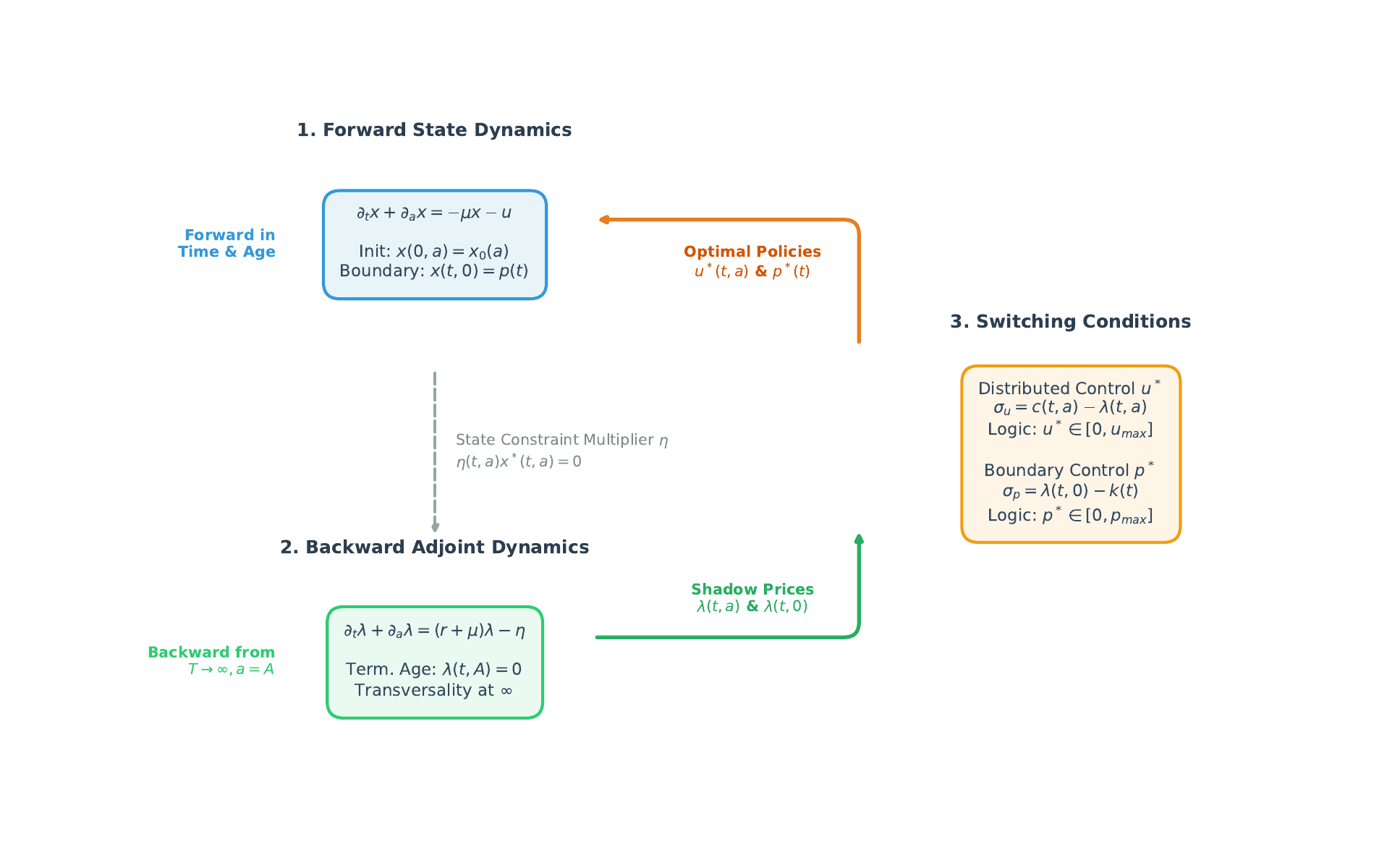}
 \caption{Information flow and system architecture of the infinite-horizon
 necessary optimality conditions. The diagram shows the coupling between the
 forward-in-time demographic state dynamics and the backward-in-time adjoint
 dynamics (with terminal data at $T\to\infty$ and $a=A$). The shadow prices
 $\lambda$ produced by the adjoint system determine the optimal distributed
 and boundary controls through the switching conditions
 \eqref{eq:u_switch}--\eqref{eq:p_switch}.}
 \label{fig:architecture}
\end{figure}

\begin{remark}[Structural contrast]
The contrast with the rate-control adjoint \eqref{eq:adjoint_theorem_PR} is
now transparent. In the rate-control case the adjoint equation is local in
age and its only source is the state-constraint multiplier. In the
effort-control case, aggregate dependence of $\mu$ produces the additional
\emph{nonlocal} term
\[
\int_0^A\mu_E(E(t),s)\,x(t,s)\,\lambda(t,s)\dd s,
\]
which couples all ages through the total stock $E(t)$ and is independent of
the age $a$ at which the adjoint equation is evaluated. This is the principal
analytical feature distinguishing the two mechanisms.
\end{remark}

\begin{remark}[Using \eqref{eq:formal_effort_adjoint} to compute controls]
The effort-control Hamiltonian is again linear in the controls, so the formal
optimality conditions yield a switching law for $w$ through the gradient
$e^{-rt}(c\,x-\lambda x)=e^{-rt}x(c-\lambda)$ (note the common factor $x\ge0$)
and for $p$ through $e^{-rt}(\lambda(t,0)-k)$. A forward--backward sweep as in
Section~\ref{subsec:numerical_use_R} applies, with two differences: (a) the
forward solve is nonlinear because $\mu$ depends on $E(t)$, requiring an inner
fixed-point/Newton iteration for $E$; and (b) each backward adjoint step must
evaluate the nonlocal integral
$\int_0^A\mu_E\,x\,\lambda\dd s$ at every age, raising the per-iteration cost
from $\mathcal O(N)$ to $\mathcal O(N)$ with a global coupling (a single quadrature reused
across all ages, so still $\mathcal O(N)$ per time level but with a non-local data
dependency). This is the practical price of the nonlocal term. We do not
implement the full effort-control optimiser here, consistently with the
formal status of Proposition~\ref{prop:formal_effort_adjoint}.
\end{remark}

\section{Autonomous reductions and explicit stationary formulas}
\label{sec:autonomous}

This section collects autonomous reductions and explicit formulas that
complement the general theory. Their role is not to introduce new existence
theory but to provide computable representations of the state and adjoint in
the stationary setting, and --- as detailed in
Section~\ref{subsec:stationary_role} --- to make the rate-control vs.\
effort-control contrast quantitative and to supply the closed-form profiles
used in the numerical study.

Throughout this section the coefficients are time-independent. For the
rate-control model we write
\begin{equation}\label{eq:auto_rate_data}
\mu(t,a)=\mu(a),\qquad c(t,a)=c(a),\qquad k(t)=k,
\end{equation}
and seek steady states
\begin{equation}\label{eq:auto_rate_ansatz}
x(t,a)=x(a),\qquad u(t,a)=u(a),\qquad p(t)=p.
\end{equation}
For the effort-control model,
\begin{equation}\label{eq:auto_effort_data}
\mu(E(t),a)=\mu(E,a),\qquad c(t,a)=c(a),\qquad k(t)=k,
\end{equation}
with steady states
\begin{equation}\label{eq:auto_effort_ansatz}
x(t,a)=x(a),\qquad w(t,a)=w(a),\qquad p(t)=p,\qquad E(t)=E.
\end{equation}
Under these assumptions the transport equations reduce to first-order ODEs in
age.

\begin{proposition}[Steady-state ODE, rate-control]
\label{prop:auto_rate_ode}
Under \eqref{eq:auto_rate_data}--\eqref{eq:auto_rate_ansatz}, the
rate-control state equation reduces to
\begin{equation}\label{eq:auto_rate_ode}
x'(a)=-\mu(a)x(a)-u(a),\qquad x(0)=p.
\end{equation}
\end{proposition}
\begin{proof}
Since $x(t,a)=x(a)$, $\partial_t x=0$; substituting
\eqref{eq:auto_rate_data}--\eqref{eq:auto_rate_ansatz} gives
\eqref{eq:auto_rate_ode}, and $x(t,0)=p(t)$ becomes $x(0)=p$.
\end{proof}

\begin{proposition}[Steady-state ODE, effort-control]
\label{prop:auto_effort_ode}
Under \eqref{eq:auto_effort_data}--\eqref{eq:auto_effort_ansatz}, the
effort-control state equation reduces to
\begin{equation}\label{eq:auto_effort_ode}
x'(a)=-\bigl(\mu(E,a)+w(a)\bigr)x(a),\qquad x(0)=p.
\end{equation}
\end{proposition}
\begin{proof}
With $\partial_t x=0$, direct substitution gives \eqref{eq:auto_effort_ode}
and $x(0)=p$.
\end{proof}

\begin{proposition}[Explicit stationary profile, rate-control]
\label{prop:explicit_rate_profile}
The unique solution of \eqref{eq:auto_rate_ode} is
\begin{equation}\label{eq:explicit_rate_profile}
x(a)=p\exp\!\left(-\int_0^a\mu(\xi)\dd\xi\right)
-\int_0^a\exp\!\left(-\int_s^a\mu(\xi)\dd\xi\right)u(s)\dd s,
\qquad 0\le a\le A.
\end{equation}
\end{proposition}
\begin{proof}
Write \eqref{eq:auto_rate_ode} as $x'+\mu x=-u$ and multiply by the
integrating factor $M(a)=\exp(\int_0^a\mu)$ to get
$\frac{d}{da}(Mx)=-Mu$. Integrating from $0$ to $a$ with $x(0)=p$ and
dividing by $M(a)$ gives \eqref{eq:explicit_rate_profile}.
\end{proof}

\begin{proposition}[Explicit stationary profile, effort-control]
\label{prop:explicit_effort_profile}
The unique solution of \eqref{eq:auto_effort_ode} is
\begin{equation}\label{eq:explicit_effort_profile}
x(a)=p\exp\!\left(-\int_0^a\bigl(\mu(E,\xi)+w(\xi)\bigr)\dd\xi\right),
\qquad 0\le a\le A.
\end{equation}
\end{proposition}
\begin{proof}
Equation \eqref{eq:auto_effort_ode} is separable; integrating
$x'/x=-(\mu(E,a)+w)$ from $0$ to $a$ with $x(0)=p$ gives
\eqref{eq:explicit_effort_profile}.
\end{proof}

\begin{proposition}[Adjoint ODE and explicit representation, rate-control]
\label{prop:explicit_adjoint}
Assume autonomous data and a stationary adjoint $\lambda(t,a)=\lambda(a)$,
$\eta(t,a)=\eta(a)$. Then the adjoint PDE
\eqref{eq:adjoint_theorem_PR} with $\lambda(t,A)=0$ reduces to
\begin{equation}\label{eq:auto_adjoint_ode}
\lambda'(a)=\bigl(r+\mu(a)\bigr)\lambda(a)-\eta(a),\qquad\lambda(A)=0,
\end{equation}
with unique solution
\begin{equation}\label{eq:auto_adjoint_formula}
\lambda(a)=\int_a^A\exp\!\left(-\int_a^s\bigl(r+\mu(\xi)\bigr)\dd\xi\right)
\eta(s)\dd s,\qquad 0\le a\le A.
\end{equation}
\end{proposition}
\begin{proof}
With $\partial_t\lambda=0$ the PDE reduces to \eqref{eq:auto_adjoint_ode}.
Writing it as $\lambda'-(r+\mu)\lambda=-\eta$ and using the integrating factor
$N(a)=\exp(-\int_0^a(r+\mu))$ gives $\frac{d}{da}(N\lambda)=-N\eta$.
Integrating from $a$ to $A$ with $\lambda(A)=0$ and dividing by $-N(a)$ yields
$\lambda(a)=\int_a^A\frac{N(s)}{N(a)}\eta(s)\dd s$, and
$N(s)/N(a)=\exp(-\int_a^s(r+\mu))$ gives \eqref{eq:auto_adjoint_formula}.
\end{proof}

\subsection{Role of the stationary formulas and economic interpretation}
\label{subsec:stationary_role}

We now make explicit how the stationary formulas connect to the optimal
control problem and what they mean economically.

\paragraph{They make the mechanism contrast quantitative.}
Formula \eqref{eq:explicit_rate_profile} is an \emph{affine} (Volterra-type)
functional of the harvesting control $u$: the stock is the unharvested
survival profile $p\,e^{-\int_0^a\mu}$ \emph{minus} a convolution of past
removals with the survival kernel $\exp(-\int_s^a\mu)$. Removal acts
\emph{additively} and can drive $x(a)$ to zero at finite age, which is exactly
where the state constraint $x\ge0$ and the multiplier $\eta$ of
Theorem~\ref{thm:necessary_PR} become active. By contrast,
\eqref{eq:explicit_effort_profile} is a \emph{purely multiplicative}
exponential: effort enters as an added mortality $w(\xi)$ inside the survival
exponent, so $x(a)>0$ for all finite $a$ whenever $p>0$. Thus the two
closed forms reproduce, at the level of stationary profiles, the
additive/affine/local vs.\ multiplicative/nonlinear/nonlocal dichotomy that
the adjoint analysis exhibits at the level of optimality systems.

\paragraph{They are stationary optimality objects, not merely state profiles.}
Formula \eqref{eq:auto_adjoint_formula} is the stationary shadow price
$\lambda(a)$: the discounted value, accrued from age $a$ to the maximal age
$A$, of relaxing the nonnegativity constraint, weighted by the survival-like
factor $\exp(-\int_a^s(r+\mu))$. Combined with the switching law
\eqref{eq:bangbang_R}, the pair
$\bigl(x(a),\lambda(a)\bigr)$ from
\eqref{eq:explicit_rate_profile}--\eqref{eq:auto_adjoint_formula} constitutes
a stationary solution of the necessary optimality system: harvesting is
switched on where $c(a)>\lambda(a)$ and off where $c(a)<\lambda(a)$. In this
sense the stationary formulas \emph{do} characterise the stationary
optimal-control problem, not just the open-loop state.

\paragraph{They are used directly in the numerics.}
The numerical study of Section~\ref{sec:numerics} evaluates
\eqref{eq:explicit_rate_profile}, \eqref{eq:explicit_effort_profile}, and
\eqref{eq:auto_adjoint_formula} (the latter integrated backward from $a=A$)
and compares them against finite-difference solutions, providing both a
validation of the schemes and the closed-form benchmark for the
yield/depletion comparison.

\paragraph{Economic reading of
\eqref{eq:explicit_rate_profile}--\eqref{eq:explicit_effort_profile}.}
Economically, \eqref{eq:explicit_rate_profile} says that under rate control
the manager pays a survival-discounted price for every unit removed and can,
in principle, exhaust an age class; the affine structure makes marginal yield
insensitive to current stock until the constraint binds, after which yield
saturates. Formula \eqref{eq:explicit_effort_profile} says that under effort
control the realised catch $w(a)x(a)$ is automatically throttled by the very
depletion it causes (since $x$ decays exponentially in cumulative effort),
which is the structural origin of the self-regulating, concave yield observed
numerically.

\section{Numerical results and discussion}
\label{sec:numerics}

We illustrate the analytical results numerically and first specify the discretisation.

\subsection{Parameters and numerical method}
\label{subsec:method}

Unless stated otherwise we use the parameters of
Table~\ref{tab:parameters}.

\begin{table}[htbp]
\centering
\caption{Baseline parameters used in the numerical study.}
\label{tab:parameters}
\begin{tabular}{@{}lll@{}}
\toprule
Symbol & Meaning & Value \\
\midrule
$A$ & maximal age & $10$ \\
$r$ & discount rate & $0.05$ \\
$\mu(a)$ & natural mortality (age law) & $0.01+0.005\,a$ \\
$\alpha$ & density-dependence coefficient (Model E) & $0.002$ \\
$p$ & stationary boundary inflow & $1$ \\
$[a_1,a_2]$ & harvesting window & $[3,7]$ (profiles), $[2,8]$ (sweeps) \\
$h$ & uniform harvesting intensity & $0.08$ (profiles); swept in $[0,0.5]$ \\
$N_a$ & age grid points (stationary) & $500$ \\
$N_a\times N_t$ & space--time grid (transient) & $200\times400$ \\
$u_{\max}$ & control bound (switching demo) & $0.15$ \\
$k$ & boundary inflow cost & $0.6$ \\
\bottomrule
\end{tabular}
\end{table}

\paragraph{Age discretisation.}
The interval $[0,A]$ is partitioned into $N_a$ equal cells of width
$\Delta a=A/N_a$ with nodes $a_i=i\,\Delta a$. Integrals (e.g.\ the aggregate
$E=\int_0^A x\dd a$ and the yields) are evaluated by the composite trapezoidal
rule on this grid.

\paragraph{Boundary condition.}
The inflow condition $x(t,0)=p(t)$ is imposed at the node $a_0=0$ at
every time level; for the stationary problems it sets the initial value
$x(0)=p$ of the age-ODE.

\paragraph{Transport scheme.}
For the stationary profiles the age-ODEs
\eqref{eq:auto_rate_ode}/\eqref{eq:auto_effort_ode} are integrated in $a$ by a
fourth-order Runge--Kutta march from $a=0$, which reproduces the closed forms
\eqref{eq:explicit_rate_profile}/\eqref{eq:explicit_effort_profile} to
machine-comparable accuracy and serves as a validation. For the
time-dependent rate-control run we use a first-order \emph{upwind}
finite-difference scheme for $\partial_t x+\partial_a x$ on the
$N_a\times N_t$ grid over $[0,T]$, $T=20$, with the CFL number
$\Delta t/\Delta a$ kept at or below $1$ for stability (here
$\Delta t=T/N_t$, $\Delta a=A/N_a$ give CFL $=(20/400)/(10/200)=1$).

\paragraph{Density-dependent mortality and the fixed-point solve (Model E).}
For Model E the mortality $\mu(E,a)=\mu(a)+\alpha E$ couples the profile to
its own aggregate. We resolve this by fixed-point iteration: starting from
$E^{(0)}=\int_0^A x_{\text{unharv}}\dd a$, we evaluate
\eqref{eq:explicit_effort_profile} with $E^{(m)}$, recompute
$E^{(m+1)}=\int_0^A x^{(m)}\dd a$, and iterate until
$|E^{(m+1)}-E^{(m)}|<10^{-10}$, which converges in a handful of iterations for
the parameters used (the map $E\mapsto\int x$ is a contraction for small
$\alpha$).

\paragraph{Adjoint solve and control update.}
The stationary adjoint \eqref{eq:auto_adjoint_ode} is integrated
\emph{backward} from $a=A$ with $\lambda(A)=0$ (RK4 in $-a$). Given $\lambda$,
the control is updated by the switching law \eqref{eq:bangbang_R}; for the
full coupled problem this is embedded in the forward--backward sweep of
Section~\ref{subsec:numerical_use_R}.

\paragraph{Enforcement of the state constraint.}
The nonnegativity constraint $x\ge0$ is enforced numerically by truncating the
realised extraction: at each age the removed amount is capped so that the
updated density cannot become negative, i.e.\ the effective rate is
$u_{\text{actual}}=\min\{u,\;x/\Delta a + \text{(survival inflow)}\}$ at the
discrete level. This is the numerical surrogate for the multiplier $\eta$ and
is exactly what produces the saturation of the rate-control yield reported
below.

\paragraph{Stopping criterion and grid convergence.}
Iterative solves (fixed point for $E$; sweeps for the coupled problem) stop
when successive iterates differ by less than $10^{-8}$ in the $L^1(0,A)$ norm.
We verified grid convergence by halving $\Delta a$ (from $N_a=250$ to $500$ to
$1000$): the stationary profiles and the aggregate $E$ change by less than
$0.2\%$ between the two finest grids, and the upwind transient solution
exhibits the expected first-order reduction of the error against the
characteristic solution under refinement.

\subsection{Stationary age profiles under harvesting}

Figure~\ref{fig:profiles} compares stationary age profiles for the two
mechanisms. In both, the boundary inflow is $p=1$ and harvesting is active on
$[3,7]$ with uniform intensity $h=0.08$. 
The prescribed harvesting profiles are
\begin{equation}\label{eq:test_profiles}
u(a)=h\,\mathbf 1_{[3,7]}(a)\quad(\text{Model R}),
\qquad
w(a)=h\,\mathbf 1_{[3,7]}(a)\quad(\text{Model E}),
\end{equation}
i.e.\ \emph{prescribed test profiles}, not computed optima: we fix a simple, identical nominal intensity in the
same age window for both models precisely so that the difference between the
two curves is attributable to the harvesting \emph{mechanism}
(additive removal $u$ vs.\ multiplicative effort $w$) and not to a difference
in the control. The rate-control profile is then obtained from
\eqref{eq:explicit_rate_profile} (equivalently, by RK4 integration of
\eqref{eq:auto_rate_ode}), and the effort-control profile from
\eqref{eq:explicit_effort_profile} with $E$ solved self-consistently as in
Section~\ref{subsec:method}.

\begin{figure}[htbp]
\centering
\includegraphics[width=\textwidth]{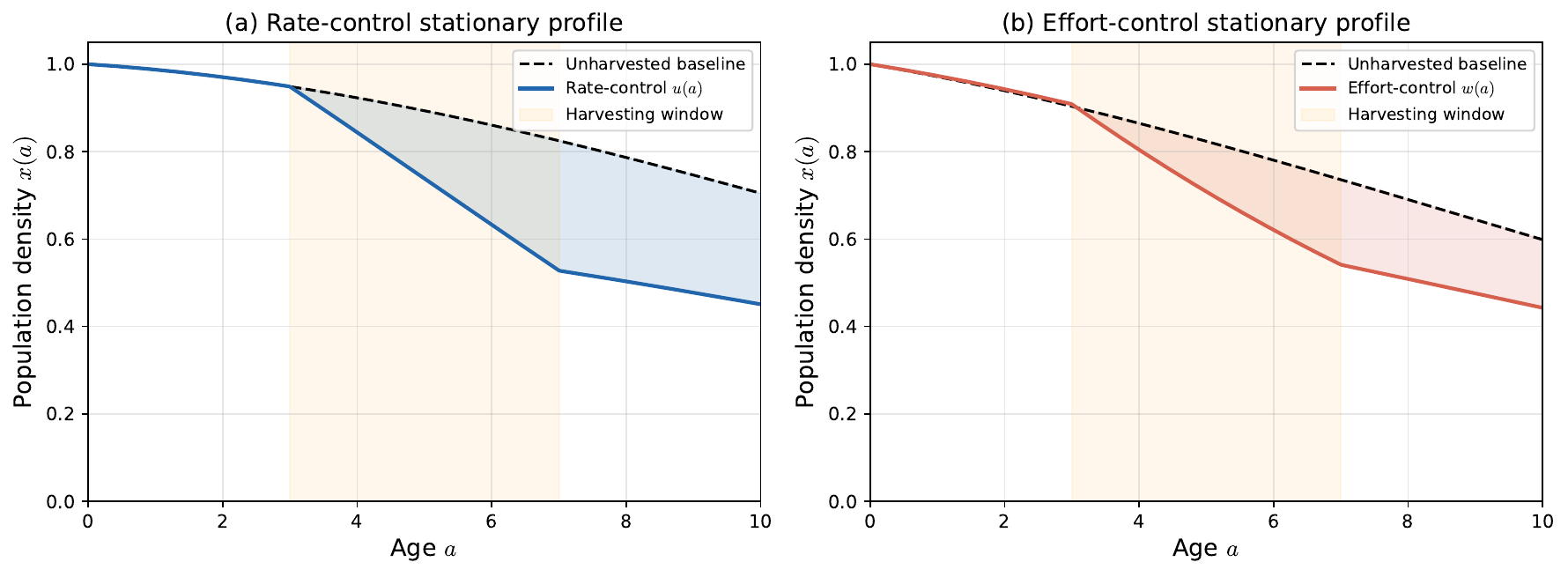}
\caption{Stationary age profiles. (a)~Rate-control: the additive removal
$u(a)=h\mathbf 1_{[3,7]}$ produces a visible change in slope at the window
boundaries $a=3,7$. (b)~Effort-control: the multiplicative effort
$w(a)=h\mathbf 1_{[3,7]}$ preserves the exponential structure and yields a
lower baseline density. Shaded region: the harvesting window $[3,7]$. Dashed
curves: the unharvested baseline ($h=0$). Both $u(a)$ and $w(a)$ are the
prescribed test profiles \eqref{eq:test_profiles}.}
\label{fig:profiles}
\end{figure}

For the effort-control model the density-dependent mortality is
$\mu(E,a)=\mu(a)+\alpha E$, $\alpha=0.002$, with $E=\int_0^A x\dd a$ solved by
fixed-point iteration. The rate-control profile shows the characteristic
slope change at $a=3,7$ (additive removal), while the effort-control profile
keeps a smooth exponential shape with a lower baseline.

\subsection{Time-dependent dynamics (rate-control)}

To show transient behaviour we solve the rate-control PDE
\eqref{eq:R_state}--\eqref{eq:R_ic} by the upwind scheme of
Section~\ref{subsec:method} on a $200\times400$ grid over $[0,20]$, with
periodic boundary inflow
\[
p(t)=0.5+0.3\sin(2\pi t/8)
\]
(seasonal restocking) and harvesting ramped linearly over $[0,5]$ to
$u=0.06$ on $[3,7]$. Figure~\ref{fig:dynamics} shows (a) the space--time
density heatmap with the diagonal transport structure, (b) age profiles at
selected times, and (c) the aggregate $E(t)$, which settles into oscillations
around a long-run mean, consistent with the periodic boundary forcing.

\begin{figure}[htbp]
\centering
\includegraphics[width=\textwidth]{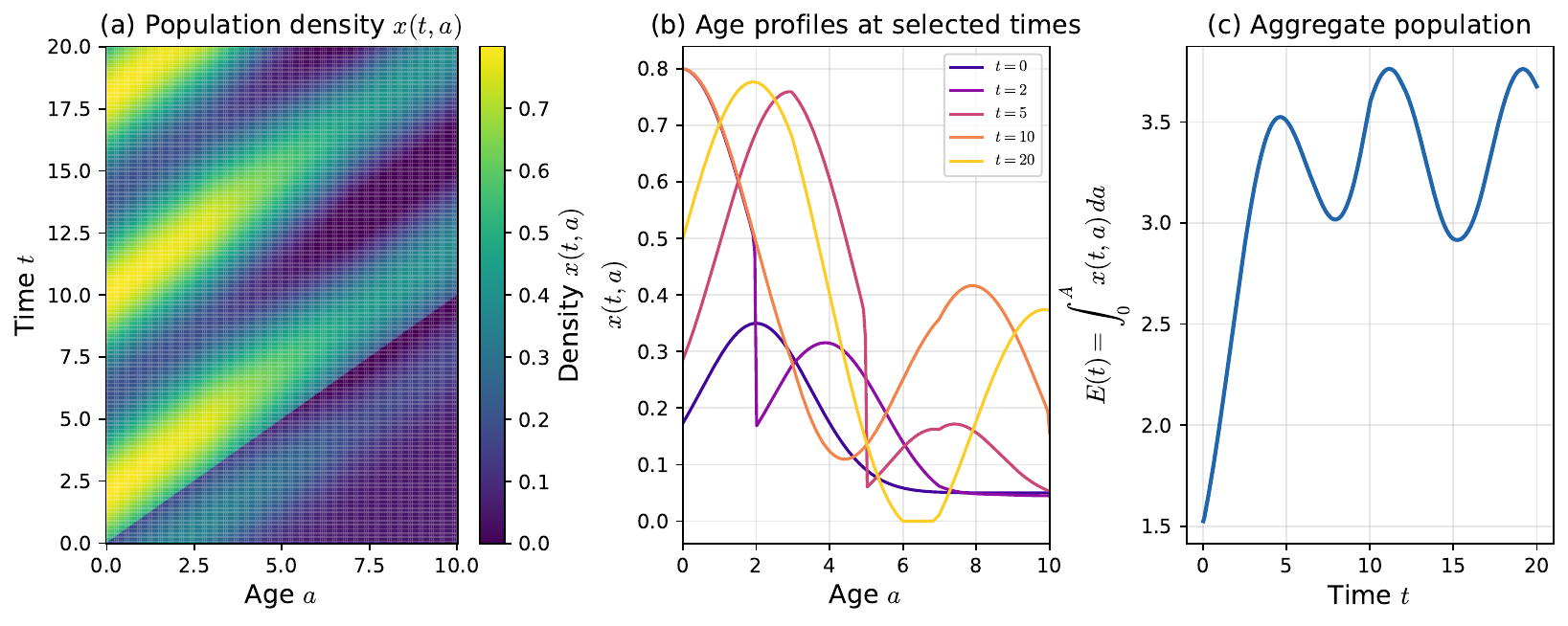}
\caption{Time-dependent dynamics of the rate-control model with periodic
boundary inflow and ramped harvesting. (a)~Space--time heatmap of the
density. (b)~Age profiles at selected times. (c)~Aggregate $E(t)$: initial
transient followed by quasi-periodic oscillations.}
\label{fig:dynamics}
\end{figure}

\subsection{Adjoint variable and switching structure}
\label{subsec:adjoint_numeric}

Figure~\ref{fig:adjoint} is an illustrative stationary example of the adjoint
and the switching law of Theorem~\ref{thm:necessary_PR}. We emphasise its
status: we do \emph{not} solve the full coupled
optimality system to optimality here. Instead, we prescribe a multiplier
$\eta(a)$, solve the stationary adjoint \eqref{eq:auto_adjoint_ode} backward
from $a=A$, and then \emph{apply} the switching law \eqref{eq:bangbang_R} to
obtain a control. The resulting control is therefore an \emph{approximate,
illustrative} control --- the ``bang-bang proxy'' --- that visualises how the
shadow price $\lambda$ drives the harvesting decision; it is the output of one
backward-adjoint-plus-switching step, not the converged optimum.

\begin{figure}[htbp]
\centering
\includegraphics[width=\textwidth]{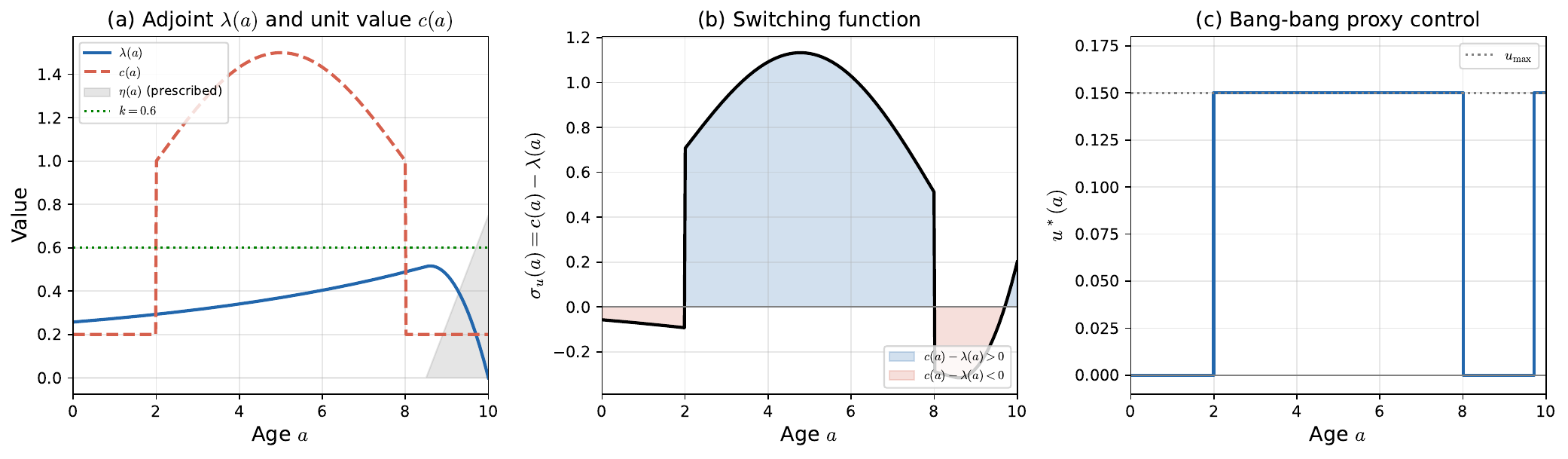}
\caption{Illustrative stationary adjoint and switching structure
(rate-control). (a)~Adjoint $\lambda(a)$, unit value $c(a)$, cost $k$, and the
prescribed multiplier $\eta(a)$ (active for $a\ge8.5$). (b)~Switching function
$\sigma_u(a)=c(a)-\lambda(a)$: positive where harvesting is favoured, negative
where conservation is favoured. (c)~The resulting bang-bang proxy control
$u^\ast(a)$, set to $u_{\max}=0.15$ where $\sigma_u$ is sufficiently positive
and to $0$ otherwise.}
\label{fig:adjoint}
\end{figure}

\paragraph{Why this $\eta(a)$.}
By complementary slackness (Theorem~\ref{thm:necessary_PR}(v)), $\eta>0$ only
where the constraint binds, $x=0$. In an age-structured population subject to
mortality and harvesting, the age classes most likely to be driven to
extinction are the \emph{oldest}, since mortality accumulates with age and the
survival factor $e^{-\int_0^a\mu}$ is smallest near $a=A$. We therefore
prescribe $\eta(a)>0$ for $a\ge 8.5$ (and $\eta\equiv0$ otherwise) as a
representative active-set pattern: it places the positive shadow value exactly
where the nonnegativity constraint is economically and demographically most
likely to bind. This is an illustrative choice consistent with complementary
slackness; it is not the multiplier of a specific computed optimum.

The unit harvesting value is taken as
\[
c(a)=1+0.5\sin\!\bigl(\pi(a-2)/6\bigr)\ \text{for }a\in[2,8],
\qquad c(a)=0.2\ \text{otherwise},
\]
and the boundary inflow cost is $k=0.6$.

Panel~(a) shows $\lambda(a)$ with $c(a)$, $k$, and $\eta(a)$; panel~(b) the
switching function $\sigma_u(a)=c(a)-\lambda(a)$; panel~(c) the bang-bang
proxy control obtained by thresholding $\sigma_u$.

\subsection{Rate-control versus effort-control: yield and depletion}

Finally, Figure~\ref{fig:comparison} compares the two mechanisms under the
\emph{same} baseline (Table~\ref{tab:parameters}) by sweeping a common
intensity $h\in[0,0.5]$ with control profiles supported on $[2,8]$, thereby
isolating the effect of the mechanism and the effect of the control bound. For Model R, $u(a)=h\,\mathbf 1_{[2,8]}(a)$ and the
stationary yield is
\[
Y_R(h)=\int_0^A u_{\text{actual}}(a)\dd a,
\]
with $u_{\text{actual}}$ the truncated extraction enforcing $x\ge0$. For
Model E, $w(a)=h\,\mathbf 1_{[2,8]}(a)$ and
\[
Y_E(h)=\int_0^A w(a)x(a)\dd a,
\]
with $x$ depending on $h$ through the steady state with
$\mu(E,a)=\mu(a)+\alpha E$, $\alpha=0.002$.

\begin{figure}[htbp]
\centering
\includegraphics[width=\textwidth]{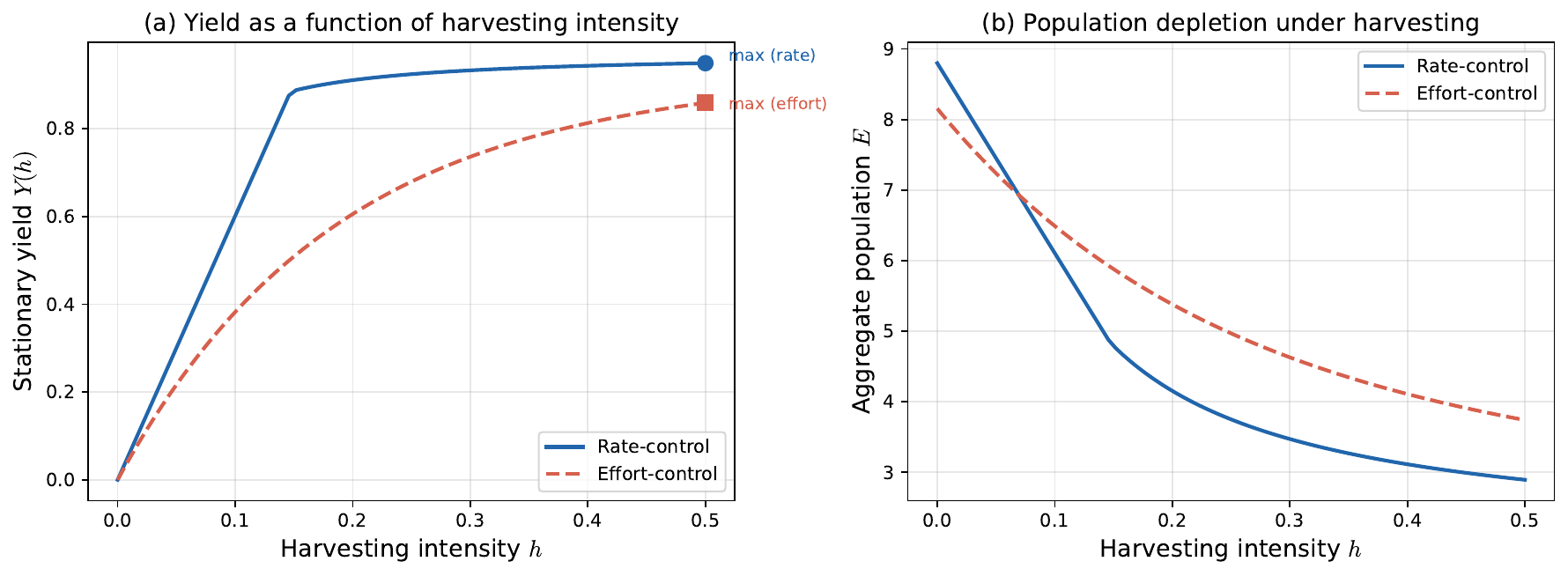}
\caption{Comparison of rate-control and effort-control harvesting under the
same baseline. (a)~Stationary yield vs.\ intensity: the rate-control yield
grows rapidly then saturates once the truncation enforcing $x\ge0$ activates
(the control bound/state constraint becomes binding), whereas the
effort-control yield is concave throughout. (b)~Aggregate depletion $E(h)$:
rate-control depletes the stock more aggressively at equal nominal intensity.}
\label{fig:comparison}
\end{figure}

\paragraph{Do the numerics support the theory?}
Yes, in three concrete ways. (i) The rate-control yield in panel~(a)
saturates precisely when the truncation enforcing $x\ge0$ activates ---
the numerical surrogate for the state-constraint multiplier $\eta$ becoming
positive --- which is the computational signature of complementary slackness
(Theorem~\ref{thm:necessary_PR}(v)). (ii) The concavity of the effort-control
yield and the milder depletion in panel~(b) are exactly what the
multiplicative profile \eqref{eq:explicit_effort_profile} predicts: catch
$w x$ is throttled by its own depletion. (iii) The stationary profiles of
Figure~\ref{fig:profiles} match the closed forms
\eqref{eq:explicit_rate_profile}/\eqref{eq:explicit_effort_profile}, and the
adjoint of Figure~\ref{fig:adjoint} matches \eqref{eq:auto_adjoint_formula},
validating the schemes. Together these confirm the additive/affine/local
vs.\ multiplicative/nonlinear/nonlocal dichotomy at the numerical level.

\section{Discussion and conclusion}
\label{sec:discussion}

We have compared two harvesting mechanisms within continuous-time
age-structured population dynamics. Juxtaposing rate-control and
effort-control, we have shown how the choice of management variable dictates
the analytical and structural properties of the resulting optimal control
problem. To respect the differing levels of rigour in the paper, we separate what is \emph{proven} from what is
\emph{formal}.

\paragraph{Proven results.}
For the rate-control model we proved existence of an optimal control for the
infinite-horizon discounted problem (Theorem~\ref{thm:existence_PR}),
exploiting the linearity of the objective in the controls, weak-$\ast$
compactness of the admissible box, and weak-$\ast$ closedness of the
state-constraint set. We established the closed-form stationary state and
adjoint representations
\eqref{eq:explicit_rate_profile}, \eqref{eq:explicit_effort_profile},
\eqref{eq:auto_adjoint_formula} (Propositions~\ref{prop:auto_rate_ode}--%
\ref{prop:explicit_adjoint}), and we gave sufficient conditions for the
infinite-horizon transversality limit (Lemma~\ref{lem:transversality_suff}).

\paragraph{Conditional results.}
The rate-control optimality system (Theorem~\ref{thm:necessary_PR}) is a
\emph{conditional} set of necessary conditions: it holds under the stated
regularity and constraint-qualification assumptions (R1)--(R7), which
postulate differentiability of the control-to-state map, the regularity (or
measure-valued nature) of the adjoint and multiplier, and a constraint
qualification. The adjoint is local in age, driven by the discount rate,
natural mortality, and the state-constraint multiplier; the controls obey
bang-bang switching laws.

\paragraph{Formal observations.}
For the effort-control model we gave a \emph{formal} derivation of the adjoint
equation (Proposition~\ref{prop:formal_effort_adjoint}). Its central feature
is the nonlocal term
\[
\int_0^A\mu_E(E(t),s)\,x(t,s)\,\lambda(t,s)\dd s,
\]
whose positive sign we verified by two independent derivations
(Section~\ref{subsec:sign_check}). This term couples the shadow prices across
all ages through the total stock. We do \emph{not} claim a maximum principle
for Model~E.

\paragraph{Bioeconomic reading.}
The analytical contrast is mirrored bioeconomically. The affine Volterra
representation \eqref{eq:explicit_rate_profile} shows that direct removal can
drive age classes to zero, requiring strict nonnegativity constraints and
favouring bang-bang strategies, with a yield that saturates once the
constraint binds. The multiplicative exponential
\eqref{eq:explicit_effort_profile} shows that effort scales extraction with
available stock, yielding a strictly concave yield curve and milder depletion
--- the classical self-regulating behaviour of effort-based management.

\paragraph{Limitations and future directions.}
First, the effort-control adjoint is formal; a rigorous infinite-horizon
maximum principle for this nonlinear, nonlocally coupled problem --- including
existence of an optimiser and the precise transversality analysis --- remains
open. 
Second, a fully unconditional rate-control maximum principle would need
the measure-valued multiplier of Remark~\ref{rem:measure_multiplier} and a
rigorous constraint qualification.
Third, our framework is deterministic;
incorporating demographic or environmental stochasticity, building on recent
well-posedness results for stochastic structured systems, is a natural next
step. Finally, extending to multi-species age-structured predator--prey
dynamics or spatial heterogeneity would let the nonlocal adjoint structures
identified here play their role in higher-dimensional bioeconomic models.

\bibliography{references}

\end{document}